\documentclass[10pt,a4paper]{article}
\usepackage{amscd}
\usepackage{amsmath,wasysym,mathrsfs}
\usepackage{graphicx}
\usepackage{amsfonts}
\usepackage{amssymb}
\usepackage{color}
\newtheorem{theorem}{Theorem}[section]
\newtheorem{lemma}[theorem]{Lemma}
\newtheorem{corollary}[theorem]{Corollary}
\newtheorem{proposition}[theorem]{Proposition}

\numberwithin{equation}{section}
\def\proof {{\noindent \bf{Proof:\hspace{4pt}}}}
\def\endproof{\hfill$\square$\vspace{6pt}}

\begin{document}

\title{Existence theory of the nonlinear plate equations}
\author{{Carlos Banquet}\\{\small Departamento de Matem\'{a}ticas y Estad\'{\i}stica, Universidad de
C\'{o}rdoba}\\{\small A.A. 354, Monter\'{\i}a, Colombia.}\\{\small \texttt{E-mail:cbanquet@correo.unicordoba.edu.co}}\vspace{.5cm}\\{{Gilmar Garbugio}}\\{\small Universidade Federal Fluminense, UFF - Departamento de
Matem\'{a}tica,} \\{\small {Rua Desembargador Ellis, 783 - Aterrado, Volta Redonda - RJ, 27213-145,
Brazil.}}\\{\small \texttt{E-mail:gilmarg@id.uff.br}} \vspace{.5cm}\\{{\'Elder J. Villamizar-Roa}{\thanks{Corresponding author.}}}\\{\small Universidad Industrial de Santander, Escuela de Matem\'{a}ticas}\\{\small A.A. 678, Bucaramanga, Colombia.} \\{\small \texttt{E-mail:jvillami@uis.edu.co}}}
\date{}
\maketitle

\begin{abstract}
This paper is devoted to the theoretical analysis of the nonlinear plate equations in $\mathbb{R}^{n}\times (0,\infty),$
$n\geq1,$ with nonlinearity involving a type polynomial behavior. We  prove the existence and uniqueness of 
global mild solutions for small initial data in $L^{1}(\mathbb{R}^{n})\cap H^s(\mathbb{R}^{n})$-spaces.
We also prove the existence and uniqueness of local and global solutions in the framework of Bessel-potential spaces $H^s_p(\mathbb{R}^n)=(I-\Delta)^{s/2}L^p(\mathbb{R}^n).$  In order to derive the existence results we develop new time decay estimates of the solution of the corresponding linear problem. 
\bigskip



\end{abstract}


\pagestyle{myheadings} \markright{Nonlinear thermoelastic plate equation}

\section{Introduction}
In recent years, the family of Cauchy problems describing deformations of elastic plates under the consideration of several physical mechanisms, including different types of interactions between sources of energy dissipation, rotational inertia and effects of nonlinearities, have caught a lot of attention by many authors \cite{Lasiecka2013,Lasiecka2017,Lasiecka2019,Racke2017}. With the aim of contributing to the theoretical development of this type of models, in this paper we consider the following nonlinear plate equation describing the evolution of a vertical displacement of a plate under the action of rotational inertia effects 
\begin{eqnarray} \label{NorPlaEqu}
u_{tt}-\mu\Delta u_{tt}+D_u\Delta^2 u-\nu^2\Delta u_t=\delta(-\Delta)^{\theta} |u|^{\lambda}, & x\in\mathbb{R}^{n},\ \ t> 0,
\end{eqnarray}
where $0\leq \theta\leq 1,$ $D_u>0,$ and $\mu,\nu,\delta\geq 0.$ The simplest submodel of (\ref{NorPlaEqu}) is given by $u_{tt}+D_u\Delta^2 u=f,$ which comes from the momentum balance equation in the description of small deflection of thin plates, under the action of a distributed transverse load $f$ acting on the plate per unit area, where the coefficient $D_u$ represents the flexural rigidity of the plate. In particular, this undamped plate equation appears as a linear model describing the vibration of stiff objects where the potential energy involves curvature-like terms which lead to the bi-laplacian operator, see Denk and Schnaubelt \cite{Denk}. In the general model (\ref{NorPlaEqu}), the term $\Delta u_{tt}$ corresponds to the rotational inertia effects, and $-\nu^2\Delta u_t$ corresponds to a dissipative term which is added to incorporate the loss of energy. Model (\ref{NorPlaEqu}) can be derived from the thermoelastic plate equations in $\mathbb{R}^n,$ $n\geq 1,$ where the heat conduction is described by the Fourier law, that is,
\begin{equation} \label{LinEqup}
\left\{
\begin{array}[c]{lc}
&u_{tt}-\mu\Delta u_{tt}+\Delta^2 u+\nu\Delta \tau=\delta(-\Delta)^{\theta} |u|^{\lambda},\\
& \tau_t-\Delta \tau-\nu\Delta u_t=0.
\end{array}
\right.  
\end{equation}
Then, neglecting the variations in time for temperature we get $\Delta \tau=-\nu\Delta u_t,$ which replacing in (\ref{LinEqup})$_1$ gives (\ref{NorPlaEqu}). Equation (\ref{NorPlaEqu}) (in the case $\delta=0$) and related models including a complete dynamic between the displacement, the thermal moment and the heat flux  has attracted the attention of researchers, and many interesting results have been obtained (see \cite{Lasiecka2017, Lasiecka2019,Racke2016,Racke2017,Racke2020} and references therein).  Depending of the choice of the involved parameters, the resulting equation (system) represents several kinds of thermoelastic plates models. Moreover,  different qualitative behaviors occur depending of the domain where the equations are defined (bounded domains, exterior domains, the half space, the whole space $\mathbb{R}^n$, etc.). In particular, in Racke and Ueda \cite{Racke2017}, by considering in (\ref{NorPlaEqu}) $\delta=0$ and $\mu=0$ and $\Delta b(\Delta u)$ in place of $\Delta^2u,$ with $b$ a given smooth function which satisfies $b'(0)>0$ and $b(0)=0,$ the authors obtained the existence of global solution $u$ in the class $(u_t,\Delta u)\in C([0,\infty);H^{s+2}(\mathbb{R}^n))$ and $u_t\in C^1([0,\infty);H^{s}(\mathbb{R}^n)),$ $s\geq [n/2]+1,$ for initial data $u(x,0)=u_0(x),$ $u_t(x,0)=u_1(x)$ satisfying $\Vert (u_1,\Delta u_0)\Vert_{H^{s+2}}$ be small enough. For $\mu>0$ and $\delta=0,$ in \cite{Racke2017} the authors proved the existence of global solution $u$ of  (\ref{NorPlaEqu}) in the class $u_t\in C([0,\infty);H^{s+2}(\mathbb{R}^n))$ with $\Delta u\in C([0,\infty);H^{s+1}(\mathbb{R}^n))$ and initial data $(u_0,u_1)$ with $\Vert \Delta u_0\Vert_{H^{s+1}}+\Vert u_1\Vert_{H^{s+2}}$ being small enough. The results of  \cite{Racke2017} were obtained by combining a local existence result with a set of {\it a priori} estimates.\\

Considering in (\ref{NorPlaEqu}) the action of a frictional displacement $u_t$ in place of $-\nu^2\Delta u_t$ and a polynomial nonlinearity (with $\theta=0$), in D'Abbicco \cite{DAbbico1} the author proved the existence of global solutions $u$ in the class $u\in C([0,\infty);H^{2}(\mathbb{R}^n))\cap C^1([0,\infty);H^{1}(\mathbb{R}^n))$ with small initial data $(u_0,u_1)\in (L^1\cap H^2)\times (L^1\cap H^1).$ The author also derives optimal estimates for $u(\cdot,t)$ in the $L^r$-norm, $r\geq 2.$ Nonlinearities of kind $\vert u_t\vert^p$ have been considered in \cite{Charao,Sugitani}.\\ 

In (\ref{NorPlaEqu}), the presence of the inertial term $-\Delta u_{tt}$ generates additional difficulties to derive decay estimates for the solution of the linear problem, in comparison with the corresponding model without inertial term. The decay of solutions to the associated linearized problem is crucial to obtain existence results, if we use a fixed point argument. These decays are usually proved by using explicit representative formula of the solution of the corresponding linear equation, as happens in the Schr\"odinger equation, or through the solution of the linear equation in terms of the Fourier transform, when we do not have the explicit formula for the inverse of the equation in Fourier variables (\cite{VB}). In our case we do not have an explicit formula for the solution of the linear problem associated to (\ref{NorPlaEqu}). In fact, the solution of the linear problem associated to (\ref{NorPlaEqu}) is given by 
\[  e^{t\varphi(\xi)}\left[\frac{\varphi(\xi)\sin(t\phi(\xi))}{\phi(\xi)} +\cos(t\phi(\xi))\right] \widehat{u_0}(\xi) +e^{t\varphi(\xi)}\frac{\sin(t\phi(\xi))}{\phi(\xi)} \widehat{ u_1}(\xi), \]
for which the phase is $\varphi(\xi)=\frac{|\xi|^2 \sqrt{3+4|\xi|^2}}{2(1+|\xi|^2)}.$ We obtain time-decay rates for the solution of the corresponding linear system, which allow us to prove the existence of global solution for (\ref{NorPlaEqu}) and $\theta=1$ in energy spaces $H^s(\mathbb{R}^n),$ for $s>\frac{n-2}{2},$ under suitable small and regular initial data, compensating the lost of regularity created by the inertial term. Explicitly, for  $s>\frac{n-2}{2},$ we prove that the solution $\partial_t S(t) u_0(x)+S(t)\Delta u_1(x)$ of the linear problem related to  (\ref{NorPlaEqu}) satisfies
\[\Vert \partial_t S(t)g \Vert_{L^{\infty}(\mathbb{R}^n)}\leq C t^{-\frac{n}{2}} \Vert g\Vert_{L^{1}(\mathbb{R}^n)}+C e^{-\frac{t}{4}}\Vert g\Vert_{H^{s+1}(\mathbb{R}^n)},\]
\[\Vert S(t) \Delta g \Vert_{L^{\infty}(\mathbb{R}^n)}\leq C t^{-\frac{n}{2}} \Vert g\Vert_{L^{1}(\mathbb{R}^n)}+C e^{-\frac{t}{4}}\Vert g\Vert_{H^{s+2}(\mathbb{R}^n)},\]
\[\Vert \partial^{k+1}_t S(t)g \Vert_{H^{s-k}(\mathbb{R}^n)}\leq C \Vert g\Vert_{H^{s}(\mathbb{R}^n)},\]
\[\Vert \partial^{k}_t S(t)\Delta g \Vert_{H^{s-k}(\mathbb{R}^n)}\leq C \Vert g\Vert_{H^{s+1}(\mathbb{R}^n)},\]
 for all $g\in \mathscr{S}(\mathbb{R}^n),$ $t> 0$, $s\in \mathbb{R}$ and $k=0,1$. 
On the other hand, in the framework of Bessel-potential spaces $H^s_p(\mathbb{R}^n)=(I-\Delta)^{s/2}L^p(\mathbb{R}^n),$ we also derive some estimates for the solutions of the corresponding linear system which allow us to obtain the existence and uniqueness of local and global solutions. Explicitly,
for $\sigma<1-n$, $2\leq p \leq  \infty,$ we obtain that the solution of the linear problem satisfies
\begin{eqnarray*}
\Vert \partial^k_t S(t)g \Vert_{H_p^{\sigma-k}(\mathbb{R}^n)}\leq C t^{-\frac n 2(1-\frac 2p)}\Vert g \Vert_{L^{p'}(\mathbb{R}^n)},
\end{eqnarray*}
for all $g\in \mathscr{S}(\mathbb{R}^n),$ $k=1,2$ and $t> 0.$  Using previous estimates we are able to study the nonlinear problem. Thus, the novelty of this work is summarized in the following aspects: First, we obtain of time-decay estimates of the solution of the linear problem in energy spaces $H^s$ and Bessel-potential $H^s_p$ spaces, as well as a set of estimates corresponding to the operator $\Lambda_{\theta}(t)=S(t)(I-\Delta)^{-1}(-\Delta)^{\theta}$ in $H^s$ and $H^s_p$ spaces. Second, we  prove existence and uniqueness of global solutions in the class $u\in C([0, \infty), H^{s}(\mathbb{R}^n)) \cap C^1([0, \infty), H^{s-1}(\mathbb{R}^n))$, for $s>\frac{n-2}{2},$ for small initial data $(u_0,u_1)$ in $L^1\cap H^s\times L^1\cap H^{s+2}$, which is weaker that the previous initial classes (cf. \cite{Racke2017}). We also prove the existence and uniqueness of local and global solutions in the the framework of Bessel-potential spaces $H^s_p(\mathbb{R}^n)$ (see Theorems \ref{teo_global_2b} and \ref{teo_local_1} below).\\

This paper is organized as follows. In Section 2, we set the main results. In Section 3, we derive time decay estimates of the linear solution in $H^s$ and $H^s_p$-spaces, as well as a set of estimates corresponding to the operator $\Lambda_{\theta}(t)=S(t)(I-\Delta)^{-1}(-\Delta)^{\theta}$. In Section 4, we prove the existence and uniqueness of global in time solutions in the framework of $H^s$ and $H^s_p$-spaces, and finally, in Section 5, we prove the existence and uniqueness of local in time solutions.
\section{Main results}
Before establishing the main results, we solve the corresponding linear problem associated to (\ref{NorPlaEqu}) which is given by
\begin{equation} \label{LinEqu}
\left\{
\begin{array}[c]{lc}
u_{tt}-\Delta u_{tt}+\Delta^2 u-\Delta u_t=0, & x\in\mathbb{R}^{n},\ \ t> 0,\\
u(x,0)=u_{0}(x),\ \ \ u_t(x,0)=\Delta u_{1}(x), & x\in\mathbb{R}^{n}.
\end{array}
\right.  
\end{equation}
Using the Fourier transform we obtain the second order differential equation
\[
(1+|\xi|^2)\widehat{u}_{tt}+|\xi|^2 \widehat{u}_t+|\xi|^4 \widehat{u}=0. 
\]
The characteristic roots of the full symbol 
\[
(1+|\xi|^2)r^2+|\xi|^2 r+|\xi|^4=0, 
\]
are given by 
\begin{align*}
r_0=r_0(\xi)=-\frac{|\xi|^2}{2(1+|\xi|^2)}+\frac{|\xi|^2 \sqrt{3+4|\xi|^2}}{2(1+|\xi|^2)}i=\varphi(\xi)+\phi(\xi)i,\\
r_1=r_1(\xi)=-\frac{|\xi|^2}{2(1+|\xi|^2)}-\frac{|\xi|^2 \sqrt{3+4|\xi|^2}}{2(1+|\xi|^2)}i=\varphi(\xi)-\phi(\xi)i.
\end{align*}
After applying the Fourier transform, we can write
\begin{align*}
\widehat{u}(\xi, t)&=\frac{r_0e^{r_1 t}-r_1e^{r_0 t}}{r_0-r_1}\widehat{u_0}(\xi) +\frac{e^{r_0t}-e^{r_1 t}}{r_0-r_1}\widehat{u_1}(\xi)\\
 & =e^{t\varphi(\xi)}\left[\frac{\varphi(\xi)\sin(t\phi(\xi))}{\phi(\xi)} +\cos(t\phi(\xi))\right] \widehat{u_0}(\xi) +e^{t\varphi(\xi)}\frac{\sin(t\phi(\xi))}{\phi(\xi)} \widehat{\Delta u_1}(\xi).
\end{align*}
Then, the global solution of the linear problem is given by
\begin{align*}
u(x,t)&=\partial_t S(t) u_0(x)+S(t)\Delta u_1(x).
\end{align*} 
Also, from the Duhamel principle, the solution of  (\ref{NorPlaEqu}) with initial data $u(x,0)=u_{0}(x), u_t(x,0)=\Delta u_{1}(x),$ is given by 
\begin{align}
u(x,t)&=\partial_t S(t) u_0(x)+S(t)\Delta u_1(x) -\int_0^t S(t-\tau)( I-\Delta)^{-1} (-\Delta)^{\theta}|u(x,\tau)|^{\lambda} d\tau,\label{e3}
\end{align} 
where 
\[
S(t)v(x)=\int_{\mathbb{R}^n}e^{-\frac{|\xi|^2 t}{2(1+|\xi|^2)}}\tfrac{2(1+|\xi|^2)}{|\xi|^2\sqrt{3+4|\xi|^2}} \sin\left(\tfrac{|\xi|^2 \sqrt{3+4|\xi|^2}t} {2(1+|\xi|^2)}\right)\widehat{v}(\xi)e^{ix\cdot\xi} d\xi,
\]
\[
\partial_t S(t)v(x)=\int_{\mathbb{R}^n} e^{-\frac{|\xi|^2 t}{2(1+|\xi|^2)}}\left[ \cos\left(\tfrac{|\xi|^2 \sqrt{3+4|\xi|^2} t }{2(1+|\xi|^2)} \right)- \tfrac{1} {\sqrt{3+4|\xi|^2}}\sin\left(\tfrac{|\xi|^2 \sqrt{3+4|\xi|^2} t}{ 2(1+|\xi|^2)}\right) \right]  \widehat{v}(\xi)e^{ix\cdot\xi}d\xi.
\] 
Henceforth we denote $\Lambda_{\theta}(t)=S(t)(I-\Delta)^{-1}(-\Delta)^{\theta}.$ Now we are in position to establish the main results of this paper. 
\begin{theorem}(Global-in-time solutions in $H^s$)\label{teo_global_2}
Let  $\lambda\geq 3,$ $\theta=1,$  and consider $s>\frac{n-2}{2}$ with $n(\lambda-2)>2.$There exists $\delta>0$ such that if 
$$C(\Vert u_0\Vert_{L^{1}}+\Vert u_0\Vert_{H^{s+1}}+\Vert u_1\Vert_{L^{1}}+\Vert u_1\Vert_{H^{s+2}})\leq \delta/2,$$
for some $C>0,$ then the initial value problem (\ref{NorPlaEqu}) has a unique global solution $u\in C([0, \infty), H^{s}(\mathbb{R}^n)) \cap C^1([0, \infty), H^{s-1}(\mathbb{R}^n))$ satisfying
$$\sup_{0<t<\infty}((1+t)^{\frac{n}{2}} \Vert u(t)\Vert_{L^\infty}+\Vert u(t)\Vert_{H^s}+\Vert u_t(t)\Vert_{H^{s-1}})\leq \delta.$$\end{theorem}
Let us also define the initial data space $\mathcal{I}_0$ as the set of pairs $[u_0,u_1]\in [\mathcal{S}'(\mathbb{R}^n)]^2$ such that the norm
\begin{eqnarray*}
\Vert [u_0,u_1]\Vert_{\mathcal{I}_0}&:=&\sup_{0<t<\infty} t^\alpha(\Vert \partial_t S(t) u_0\Vert_{H^s_p}+\Vert S(t)\Delta u_1\Vert_{H^s_p})\\
&&+\sup_{0<t<\infty}t^\beta(\Vert \partial^2_{t} S(t) u_0\Vert_{H^{s-1}_p}+\Vert \partial_tS(t)\Delta u_1\Vert_{H^{s-1}_p})<\infty,
\end{eqnarray*}
 with $\alpha=\frac{1}{\lambda-1} [ 2-\theta-\frac n2(1-\frac 2p) ]$ and $\beta=\alpha+1-\theta,$ $\lambda \geq 2,$ $\theta \in (\frac{2-n}{2},1]$ if $n=1,2$ and $\theta\in [0,1]$ if $n\geq 3.$
 We also consider the norm,
\begin{eqnarray*}
\Vert u\Vert_{\mathcal{X}^{s,p}_{\alpha, \beta}}&:=&\sup_{0<t<\infty}( t^{\alpha} \Vert u(t)\Vert_{H^s_p}+ t^{\beta}\Vert u_t(t)\Vert_{H_p^{s-1}}).
\end{eqnarray*}

\begin{theorem}(Global-in-time solutions in $H^s_p$)\label{teo_global_2b}
Let $\lambda \geq 2,$ $\theta \in (\frac{2-n}{2},1]$ if $n=1,2$ and $\theta\in [0,1]$ if $n\geq 3,$ and $\frac{1}{\lambda}>\alpha>0.$ Assume $2\leq p \leq q\leq  \infty,$ $\frac n2(1-\frac 2p)<1$ and consider $\sigma,s$ such that $s>\sigma,$ $n(\frac{1}{p}-\frac{1}{q})\leq \sigma<3-n-2\theta$ and $(\frac{\lambda}{q}+\frac{1}{p}-1)\frac{n}{\lambda-1}+\sigma\leq s<\min\{\frac{n}{q},\lambda-1\}+\sigma.$ There exists $\delta>0$ such that if 
$$\Vert [u_0,u_1]\Vert_{\mathcal{I}_0}\leq \delta/2,$$
then the initial value problem (\ref{NorPlaEqu}) has a unique global solution $u\in C([0,\infty):H^s_p(\mathbb{R}^n))\cap C^1([0,\infty);H^{s-1}_p(\mathbb{R}^n))$ satisfying $\Vert u\Vert_{\mathcal{X}^{s,p}_{\alpha,\beta}}\leq {\delta}.$
\end{theorem}

\begin{theorem}(Local-in-time solutions)\label{teo_local_1} Let $\lambda \geq 2,$ $\theta \in (\frac{2-n}{2},1]$ if $n=1,2$ and $\theta\in [0,1]$ if $n\geq 3.$ Assume $2\leq p \leq q\leq  \infty,$ and consider $\sigma,s$ such that $s>\sigma,$ $n(\frac{1}{p}-\frac{1}{q})\leq \sigma<3-n-2\theta$ and $(\frac{\lambda}{q}+\frac{1}{p}-1)\frac{n}{\lambda-1}+\sigma\leq s<\min\{\frac{n}{q},\lambda-1\}+\sigma,$ and $1>\frac{n}{2}(1-\frac{2}{p})\lambda.$  Then, if $[u_0,u_1]\in  {H^{s+1-\sigma}_{p'}}\times  {H^{s-\sigma}_{p'}},$
 there exists $0<T<\infty$ such that the initial value problem (\ref{NorPlaEqu}) has a unique local in time solution $u\in C([0,T];H^s_p(\mathbb{R}^n))\cap C^1([0,T];H^{s-1}_p(\mathbb{R}^n)).$
\end{theorem}

\section{Time decay estimates in $H^s$ and $H^s_p.$}
The first aim of this section is to derive some decay estimates of the semigroups $\Lambda_{\theta}(t),$ $\partial_t \Lambda_{\theta}(t),$ $S(t)$ and $\partial_t S(t)$ on $L^{\infty}(\mathbb{R}^n),$ $H^s(\mathbb{R}^n)$ and  $H^s_p(\mathbb{R}^n)$ spaces. 
\subsection{Estimates in $L^\infty(\mathbb{R}^n)$ and $H^s(\mathbb{R}^n).$}
\begin{lemma}\label{AcotGamma}
Let  $a>-n$. There exists a constant $C_n,$ that only depends on $n,$ such that 
\[ \int_{|\xi|\leq 1}e^{-\frac{|\xi|^2 t}{4}} |\xi|^{a}d\xi \leq C_n 2^{a+n-1} t^{-\frac{n+a}{2}}.\]
\end{lemma}
\proof
Making the change of variable $\eta=\frac{\xi \sqrt{t}}{2},$ we obtain 
\begin{equation}\label{Gamma1}
 \int_{|\xi|\leq 1}e^{-\frac{|\xi|^2 t}{4}} |\xi|^{a}d\xi = 2^{a+n} t^{-\frac{n+a}{2}} \int_{|\eta|\leq \frac{\sqrt{t}}{2}} e^{-|\eta|^2} |\eta|^{a}d\eta \leq  2^{a+n} t^{-\frac{n+a}{2}} \int_{\mathbb{R}^n} e^{-|\eta|^2} |\eta|^{a}d\eta.
\end{equation}
Now, using spherical coordinates in $\mathbb{R}^n$ and the change of variable $r=\sqrt{x},$ it yields 
\begin{equation}\label{Gamma2}
\int_{\mathbb{R}^n} e^{-|\eta|^2} |\eta|^{a}d\eta = C_n \int_0^{\infty} e^{-r^2}r^{a+n-1}dr=  \frac{C_n}{2} \int_0^{\infty} e^{-x}x^{\frac{a+n}{2}-1}dx = \frac {C_n}{2} \ \Gamma\left(\frac{n+a}{2}\right), 
\end{equation}
where $\Gamma$ is the Gamma function. From the assumption $a>-n,$ it holds that $\Gamma(\frac{n+a}{2})$ is finite because $\mathrm{Re} (\frac{n+a}{2})>0.$ Combining (\ref{Gamma1}) and (\ref{Gamma2}), we obtain the desired result.
\endproof

At this point, using the previous lemma, we can establish estimates in the norm $L^{\infty}$ for the linear operator $\Lambda_{\theta}(t)$ appearing in (\ref{e3}). For any positive elements $A=A(\xi,t)$ and $B=B(\xi,t),$ the notation $A\apprle B,$ means that there exists a positive constant $c,$ which does not depend on $t$ or $\xi,$ such that $A\leq c B.$
\begin{lemma}\label{HsL1The<=1}
Let  $\theta \in (\frac{2-n}{2},1]$ if $n=1,2$ and $\theta\in [0,1]$ if $n\geq 3.$ Assume that $s>\frac{n+4\theta-6}{2}$. Then, there is a constant $C=C_{n}>0$ such that
\[\Vert \Lambda_{\theta} (t)g \Vert_{L^{\infty}(\mathbb{R}^n)}\leq C t^{-\frac{n+2(\theta-1)}{2}}    \Vert g\Vert_{L^{1}(\mathbb{R}^n)}+C e^{-\frac{t}{4}}\Vert g\Vert_{H^{s}(\mathbb{R}^n)},\]
for all $g\in L^1(\mathbb{R}^n)\cap H^s(\mathbb{R}^n)$ and $t> 0,$ being  $\Lambda_{\theta}(t)=S(t)(I-\Delta)^{-1}(-\Delta)^{\theta}.$
\end{lemma}
\proof
Since  the $\sin(\cdot)$ function is bounded, $\tfrac{2}{\sqrt{3+4|\xi|^2}}  \apprle 1$ and  $e^{-\frac{|\xi|^2 t}{2(1+|\xi|^2)}} \leq e^{- \frac t4}$ for $|\xi| \geq 1,$ we obtain
\begin{align*}
|\Lambda_{\theta}(t)g(x)|& \leq \left| \int_{|\xi|\leq 1}e^{-\frac{|\xi|^2 t}{2(1+|\xi|^2)}} \tfrac{2|\xi|^{2(\theta-1)}}{\sqrt{3+4|\xi|^2}} \sin\left(\tfrac{|\xi|^2 \sqrt{3+4|\xi|^2}t} {2(1+|\xi|^2)}\right)\widehat{g}(\xi)e^{ix\cdot\xi} d\xi\right|\\
& +\left| \int_{|\xi|\geq 1}e^{-\frac{|\xi|^2 t}{2(1+|\xi|^2)}}\tfrac{2|\xi|^s|\xi|^{-s} |\xi|^{2(\theta-1)} }{ \sqrt{3+4|\xi|^2}} \sin\left(\tfrac{|\xi|^2 \sqrt{3+4|\xi|^2}t} {2(1+|\xi|^2)}\right)\widehat{g}(\xi)e^{ix\cdot\xi} d\xi\right|\\
& \apprle   \| \widehat{g}\|_{L^{\infty}}   \int_{|\xi|\leq 1}e^{-\frac{|\xi|^2 t}{2(1+|\xi|^2)}} |\xi|^{2(\theta-1)} d\xi +  e^{-\frac t4}  \int_{|\xi|\geq 1}\tfrac{2|\xi|^s|\xi|^{-s} |\xi|^{2(\theta-1)} }{ \sqrt{3+4|\xi|^2}} | \widehat{g}(\xi) |d\xi.
\end{align*}
Taking into account that $ \| \widehat{g}\|_{L^{\infty}(\mathbb{R}^n)} \leq    \|g\|_{L^1(\mathbb{R}^n)} $ and $e^{-\frac{|\xi|^2 t}{2(1+|\xi|^2)}} \leq e^{- \frac{|\xi|^2t}{4}}$ for $|\xi| \leq 1,$ we get
\begin{align*}
|\Lambda_{\theta}(t)g(x)| & \apprle   \|g\|_{L^1(\mathbb{R}^n)} \int_{|\xi|\leq 1}e^{-\frac{|\xi|^2 t}{4}} |\xi|^{2(\theta-1)} d\xi + e^{-\frac {t}{4}} \int_{|\xi|\geq 1}\tfrac{2|\xi|^s|\xi|^{-s} |\xi|^{2(\theta-1)}}{ \sqrt{3+4|\xi|^2}}  |\widehat{g}(\xi)|  d\xi.
\end{align*}
Now, from Lemma \ref{AcotGamma} and using the Cauchy-Schwarz inequality, we have
\begin{align*}
|\Lambda_{\theta}(t)g(x)|  & \apprle   t^{-\frac{n+2(\theta-1)}{2}}  \|g\|_{L^1(\mathbb{R}^n)}  + e^{-\frac {t}{4}} \|g\|_{H^s(\mathbb{R}^n)} \left( \int_{|\xi|\geq 1}\tfrac{4|\xi|^{-2s} |\xi|^{4(\theta-1)}}{3+4|\xi|^2} d\xi\right)^{1/2}.
\end{align*}
By applying spherical coordinates and since $-2s+4\theta -6+n<0,$ we obtain
\begin{align*}
|\Lambda_{\theta}(t)g(x)|  & \apprle   t^{-\frac{n+2(\theta-1)}{2}}  \|g\|_{L^1(\mathbb{R}^n)}  + e^{-\frac {t}{4}} \|g\|_{H^s(\mathbb{R}^n)} \left( \int_1^{\infty} r^{-2s+4\theta+n-7} dr\right)^{1/2} \\
& \apprle  t^{-\frac{n+2(\theta-1)}{2}}   \|g\|_{L^1(\mathbb{R}^n)} + e^{-\frac {t}{4}} \|g\|_{H^s(\mathbb{R}^n)},
\end{align*}
which finishes the proof of the lemma. 
\endproof

In order to deal with the existence of solutions in the energy spaces $H^s(\mathbb{R}^n),$ we need to obtain time decay estimates for $\Lambda_{\theta}(t)$ of kind $(1+t)^{-\frac{n+2(\theta-1)}{2}}$ in place of $t^{-\frac{n+2(\theta-1)}{2}}.$ Otherwise, we would need stronger restrictions on the parameters $\lambda$ and $\theta$ that make impossible to have a nonempty set of constraints for the existence of mild solutions.
\begin{lemma}\label{HsL1(1+t)<=1}
Let   $\theta \in (\frac{2-n}{2},1]$ if $n=1,2$ and $\theta\in [0,1]$ if $n\geq 3.$  Consider $s>\frac{n+4\theta -6}{2}$. Then, there is a constant $C=C_{n,\theta}>0,$ that only depends on $n$ and $\theta,$ such that
\[\Vert \Lambda_{\theta} (t)g \Vert_{L^{\infty}(\mathbb{R}^n)}\leq C (1+t)^{-\frac{n+2(\theta-1)}{2}} \left[ \Vert g\Vert_{L^{1}(\mathbb{R}^n)}+ \Vert g\Vert_{H^{s}(\mathbb{R}^n)} \right],\]
for all $g\in L^1(\mathbb{R}^n)\cap H^s(\mathbb{R}^n)$ and $t> 0,$ being  $\Lambda_{\theta}(t)=S(t)(I-\Delta)^{-1}(-\Delta)^{\theta}.$
\end{lemma}
\proof
We assume first that $t>1.$ Since $n+2(\theta-1)>0$ and 
\[ t^{-\frac{n+2(\theta-1)}{2}}=(2t)^{-\frac{n+2(\theta-1)}{2}} 2^{\frac{n+2(\theta-1)}{2}}=(t+t)^{-\frac{n +2(\theta-1)}{2}}2^{\frac{n+2(\theta-1)}{2}}\leq (1+t)^{-\frac{n+2(\theta-1)}{2}}2^{\frac{n+2(\theta-1)}{2}}, \]
 from Lemma \ref{HsL1The<=1}, we arrive at
\[\Vert \Lambda_{\theta} (t)g \Vert_{L^{\infty}(\mathbb{R}^n)}\leq C (1+t)^{-\frac{n+2(\theta-1)}{2}} \Vert g\Vert_{L^{1}(\mathbb{R}^n)}+ C e^{-\frac t4}\Vert g\Vert_{H^{s}(\mathbb{R}^n)}.\]
Using the last inequality and the fact that there exists a constant $C_{n,\theta}>0,$ which only depends on $n$ and $\theta,$ such that $e^{-\frac t4} \leq C_{n,\theta}(1+t)^{-\frac {n+2(\theta-1)}{2}},$ we obtain the desired result for $t>1$.\\

Now, we consider the case $0< t \leq  1.$ Since $\tfrac{2 (1+|\xi|^2) }{t |\xi|^2 \sqrt{3+4|\xi|^2}} \left|   \sin\left(\tfrac{|\xi|^2 \sqrt{3+4|\xi|^2}t} {2(1+|\xi|^2)}\right)  \right| \leq 1$ and,  $e^{-\frac{|\xi|^2 t}{2(1+|\xi|^2)}} \leq e^{- \frac t4}$ for $|\xi| \geq 1,$ we have 
\begin{align*}
|\Lambda_{\theta}(t)g(x)|& \leq t \left| \int_{|\xi|\leq 1} e^{-\frac{|\xi|^2 t}{2(1+|\xi|^2)}} \tfrac{|\xi|^{2\theta}}{1+|\xi|^2} \tfrac{2 (1+|\xi|^2) }{t |\xi|^2 \sqrt{3+4|\xi|^2}} \sin\left(\tfrac{|\xi|^2 \sqrt{3+4|\xi|^2}t} {2(1+|\xi|^2)}\right) \widehat{g}(\xi)e^{ix\cdot\xi} d\xi\right| \notag\\
& +\left| \int_{|\xi|\geq 1}e^{-\frac{|\xi|^2 t}{2(1+|\xi|^2)}}\tfrac{2|\xi|^s|\xi|^{-s} |\xi|^{2(\theta-1)} }{ \sqrt{3+4|\xi|^2}} \sin\left(\tfrac{|\xi|^2 \sqrt{3+4|\xi|^2}t} {2(1+|\xi|^2)}\right)\widehat{g}(\xi)e^{ix\cdot\xi} d\xi\right|\\
& \apprle  t \| \widehat{g}\|_{L^{\infty}}   \int_{|\xi|\leq 1} e^{-\frac{|\xi|^2 t}{2(1+|\xi|^2)}}   \tfrac{|\xi|^{2\theta}}{1+|\xi|^2}  d\xi + e^{-\frac{t}{4}}  \int_{|\xi|\geq 1}\tfrac{2|\xi|^s|\xi|^{-s} |\xi|^{2(\theta-1)} }{ \sqrt{3+4|\xi|^2}} | \widehat{g}(\xi) |d\xi.
\end{align*}
Using that $ e^{-\frac{|\xi|^2 t}{2(1+|\xi|^2)}}   \tfrac{|\xi|^{2\theta}}{1+|\xi|^2} \leq 1,$ the Cauchy-Schwarz inequality, spherical coordinates and taking into account that $-2s+4\theta+n-6<0$, we arrive at
\begin{align*}
|\Lambda_{\theta}(t)g(x)| & \apprle  t \| \widehat{g}\|_{L^{\infty}}   \int_{|\xi|\leq 1} 1  d\xi + e^{-\frac{t}{4}}  \|g\|_{H^s(\mathbb{R}^n)}  \left( \int_{|\xi|\geq 1}\tfrac{4|\xi|^{-2s} |\xi|^{4(\theta-1)} }{3+4|\xi|^2} |d\xi \right)^{\frac 12}\\
& \apprle  t \|g\|_{L^1(\mathbb{R}^n)} + e^{-\frac {t}{4}} \|g\|_{H^s(\mathbb{R}^n)}.
\end{align*}
Again, since there exists a constant $C_{n,\theta},$ such that $e^{-\frac t4} \leq C_{n,\theta}(1+t)^{-\frac {n+2(\theta-1)}{2}},$ we obtain
\begin{equation}\label{Lambda(1+t)}
|\Lambda_{\theta}(t)g(x)| \apprle  t \|g\|_{L^1(\mathbb{R}^n)} +  (1+t)^{-\frac{n+2(\theta-1)}{2}} \|g\|_{H^s(\mathbb{R}^n)}.
\end{equation}
Also, since $t \leq 1,$ we have
\begin{align}\label{gL1(1+t)} 
 t\|g\|_{L^1(\mathbb{R}^n)} \leq  \|g\|_{L^1(\mathbb{R}^n)} =& 2^{\frac{n+2(\theta-1)}{2}} 2^{-\frac{n+2(\theta-1)}{2}} \|g\|_{L^1(\mathbb{R}^n)} = 2^{\frac{n+2(\theta-1)}{2}} (1+1)^{-\frac{n+2(\theta-1)}{2}} \|g\|_{L^1(\mathbb{R}^n)} \notag\\
 & \leq 2^{\frac{n+2(\theta-1)}{2}} (1+t)^{-\frac{n+2(\theta-1)}{2}} \|g\|_{L^1(\mathbb{R}^n)} \apprle  (1+t)^{-\frac{n+2(\theta-1)}{2}}\|g\|_{L^1(\mathbb{R}^n)}.
\end{align}
Combining (\ref{Lambda(1+t)}) and (\ref{gL1(1+t)}), we obtain the desired result.
\endproof

The following lemma helps to deal with the initial data of the integro-differential equation (\ref{e3}) and its derivative  on the variable $t.$ 
\begin{lemma}\label{DatIni1}
Let  $s>\frac{n-2}{2}$, there exists a constant $C=C_n>0,$ which only depends on $n,$ such that
\[\Vert \partial_t S(t)g \Vert_{L^{\infty}(\mathbb{R}^n)}\leq C t^{-\frac{n}{2}} \Vert g\Vert_{L^{1}(\mathbb{R}^n)}+C e^{-\frac{t}{4}}\Vert g\Vert_{H^{s+1}(\mathbb{R}^n)},\]
\[\Vert S(t) \Delta g \Vert_{L^{\infty}(\mathbb{R}^n)}\leq C t^{-\frac{n}{2}} \Vert g\Vert_{L^{1}(\mathbb{R}^n)}+C e^{-\frac{t}{4}}\Vert g\Vert_{H^{s+2}(\mathbb{R}^n)},\]
\[\Vert \partial^{k+1}_t S(t)g \Vert_{H^{s-k}(\mathbb{R}^n)}\leq C \Vert g\Vert_{H^{s}(\mathbb{R}^n)},\]
\[\Vert \partial^{k}_t S(t)\Delta g \Vert_{H^{s-k}(\mathbb{R}^n)}\leq C \Vert g\Vert_{H^{s+1}(\mathbb{R}^n)},\]
 for all $g\in \mathscr{S}(\mathbb{R}^n),$ $t> 0$, $s\in \mathbb{R}$ and $k=0,1$.  
\end{lemma}
\proof
First, we prove the estimates on $L^{\infty}.$ From the boundedness of $\cos(\cdot)$ and $\sin(\cdot)$ functions, and since $ \tfrac{1}{ \sqrt{3+4|\xi|^2}} \leq 1,$ we obtain
\begin{align*}
|\partial_t S(t) g(x)|& \leq \left| \int_{|\xi|\leq 1}e^{-\frac{|\xi|^2 t}{2(1+|\xi|^2)}}  \cos \left(\tfrac{|\xi|^2 \sqrt{3+4|\xi|^2}t} {2(1+|\xi|^2)}\right)\widehat{g}(\xi)e^{ix\cdot\xi} d\xi\right|\\
&+\left| \int_{|\xi|\geq 1}e^{-\frac{|\xi|^2 t}{2(1+|\xi|^2)}}  |\xi|^{s+1} |\xi|^{-(s+1)} \cos \left(\tfrac{|\xi|^2 \sqrt{3+4|\xi|^2}t} {2(1+|\xi|^2)}\right)\widehat{g}(\xi)e^{ix\cdot\xi} d\xi\right|\\
& +\left| \int_{|\xi|\leq  1}e^{-\frac{|\xi|^2 t}{2(1+|\xi|^2}} \tfrac{1}{ \sqrt{3+4|\xi|^2}} \sin\left(\tfrac{|\xi|^2 \sqrt{3+4|\xi|^2}t} {2(1+|\xi|^2)}\right)  \widehat{g}(\xi)e^{ix\cdot\xi} d\xi\right|\\
&+\left| \int_{|\xi|\geq 1}e^{-\frac{|\xi|^2 t}{2(1+|\xi|^2)}}\tfrac{|\xi|^{s}|\xi|^{-s}}{ \sqrt{3+4|\xi|^2}} \sin\left(\tfrac{|\xi|^2 \sqrt{3+4|\xi|^2}t} {2(1+|\xi|^2)}\right)\widehat{g}(\xi)e^{ix\cdot\xi} d\xi\right|\\
& \apprle  \| \widehat{g}\|_{L^{\infty}}  \int_{|\xi|\leq 1}e^{-\frac{|\xi|^2 t}{2(1+|\xi|^2)}}  d\xi + \int_{|\xi|\geq 1} e^{-\frac{|\xi|^2 t}{2(1+|\xi|^2)}}  |\xi|^{s+1}|\xi|^{-(s+1)}  |\widehat{g}(\xi)|  d\xi\\
& +\int_{|\xi|\geq 1}  e^{-\frac{|\xi|^2 t}{2(1+|\xi|^2)}}  \tfrac{ |\xi|^{s}|\xi|^{-s}}{ \sqrt{3+4|\xi|^2}}  |\widehat{g}(\xi)|  d\xi.
\end{align*}
Since $ \| \widehat{g}\|_{L^{\infty}} \leq    \|g\|_{L^1(\mathbb{R}^n)}, $  $e^{-\frac{|\xi|^2 t}{2(1+|\xi|^2)}} \leq e^{- \frac t4}$ for $|\xi| \geq 1$ and,  $e^{-\frac{|\xi|^2 t}{2(1+|\xi|^2)}} \leq e^{- \frac{|\xi|^2t}{4}}$ for $|\xi| \leq 1,$ we have
\begin{align*}
|\partial_t S(t) g(x)| & \apprle   \|g\|_{L^1(\mathbb{R}^n)}\int_{|\xi|\leq 1}e^{-\frac{|\xi|^2 t}{4}}  d\xi  +e^{-\frac {t}{4}} \int_{|\xi|\geq 1} |\xi|^{s+1}|\xi|^{-(s+1)}  |\widehat{g}(\xi)|  d\xi +e^{-\frac {t}{4}} \int_{|\xi|\geq 1}\tfrac{ |\xi|^{s}|\xi|^{-s}}{ \sqrt{3+4|\xi|^2}}  |\widehat{g}(\xi)|  d\xi.
\end{align*}
By Lemma \ref{AcotGamma}, the Cauchy-Schwarz inequality and since $-2s+n-2<0$, we conclude that
\begin{align*}
|\partial_t S(t) g(x)| & \apprle  t^{-\frac{n}{2}} \|g\|_{L^1(\mathbb{R}^n)} + e^{-\frac {t}{4}} \|g\|_{H^{s+1}(\mathbb{R}^n)} \left( \int_{|\xi|\geq 1} |\xi|^{-2(s+1)} d\xi\right)^{1/2} \\
&+ e^{-\frac {t}{4}} \|g\|_{H^{s}(\mathbb{R}^n)} \left( \int_{|\xi|\geq 1}\tfrac{ |\xi|^{-2s}}{3+4|\xi|^2} d\xi\right)^{1/2} \\
& \apprle  t^{-\frac{n}{2}} \|g\|_{L^1(\mathbb{R}^n)} + e^{-\frac {t}{4}} \|g\|_{H^{s+1}(\mathbb{R}^n)}.
\end{align*}
Therefore, 
\[\Vert \partial_tS(t)g \Vert_{L^{\infty}(\mathbb{R}^n)}\leq C t^{-\frac{n}{2}} \Vert g\Vert_{L^{1}(\mathbb{R}^n)}+C e^{-\frac{t}{4}}\Vert g\Vert_{H^{s+1}(\mathbb{R}^n)}.\]
which proves the first inequality of the lemma.\\ 

To obtain the second inequality, we use a similar argument as before. Indeed,
\begin{align*}
|S(t)\Delta g(x)|& \leq \left| \int_{|\xi|\leq 1}e^{-\frac{|\xi|^2 t}{2(1+|\xi|^2)}}\tfrac{2(1+|\xi|^2)}{\sqrt{3+4|\xi|^2}} \sin\left(\tfrac{|\xi|^2 \sqrt{3+4|\xi|^2}t} {2(1+|\xi|^2)}\right)\widehat{g}(\xi)e^{ix\cdot\xi} d\xi\right|\\
& +\left| \int_{|\xi|\geq 1}e^{-\frac{|\xi|^2 t}{2(1+|\xi|^2)}}\tfrac{2(1+|\xi|^2)|\xi|^{s+2}|\xi|^{-(s+2)}}{ \sqrt{3+4|\xi|^2}} \sin\left(\tfrac{|\xi|^2 \sqrt{3+4|\xi|^2}t} {2(1+|\xi|^2)}\right)\widehat{g}(\xi)e^{ix\cdot\xi} d\xi\right|\\
& \apprle  t^{-\frac{n}{2}} \|g\|_{L^1(\mathbb{R}^n)} + e^{-\frac {t}{4}} \int_{|\xi|\geq 1}\tfrac{2(1+|\xi|^2)|\xi|^{s+2}|\xi|^{-(s+2)}}{ \sqrt{3+4|\xi|^2}}  |\widehat{g}(\xi)|  d\xi \\
& \apprle  t^{-\frac{n}{2}} \|g\|_{L^1(\mathbb{R}^n)}+ e^{-\frac {t}{4}} \|g\|_{H^{s+2}(\mathbb{R}^n)} \left( \int_{|\xi|\geq 1}\tfrac{4(1+|\xi|^2)^2|\xi|^{-2(s+2)}}{3+4|\xi|^2} d\xi\right)^{1/2} \\
& \apprle  t^{-\frac{n}{2}} \|g\|_{L^1(\mathbb{R}^n)}+ e^{-\frac {t}{4}} \|g\|_{H^{s+2}(\mathbb{R}^n)}.
\end{align*}
which finishes the proof for the operator $S(t)\Delta$ in $L^{\infty}.$\\

The third inequality, for $k=0,$ follows by observing that
\begin{align*}
\|\partial_t S(t) g\|^2_{H^s}& \leq  \int_{\mathbb{R}^n} \left| e^{-\frac{|\xi|^2 t}{2(1+|\xi|^2)}}  (1+|\xi|^2)^{\frac s2} \cos \left(\tfrac{|\xi|^2 \sqrt{3+4|\xi|^2}t} {2(1+|\xi|^2)}\right)\widehat{g}(\xi) \right|^2 d\xi\\\
& + \int_{\mathbb{R}^n} \left| e^{-\frac{|\xi|^2 t}{2(1+|\xi|^2}} \tfrac{(1+|\xi|^2)^{\frac s2}}{ \sqrt{3+4|\xi|^2}} \sin\left(\tfrac{|\xi|^2 \sqrt{3+4|\xi|^2}t} {2(1+|\xi|^2)}\right)  \widehat{g}(\xi) \right|^2 d\xi\\
 &\leq  \int_{\mathbb{R}^n} \left| (1+|\xi|^2)^{\frac s2} \widehat{g}(\xi) \right|^2 d\xi + \int_{\mathbb{R}^n} \left| (1+|\xi|^2)^{\frac{s-1}{2}} \widehat{g}(\xi) \right|^2 d\xi\\
 &  \leq \| g\|^2_{H^s} +  \| g\|^2_{H^{s-1}}  \leq \| g\|^2_{H^s}.
\end{align*}
Therefore,
\[
\|\partial_t S(t) g|_{H^s} \leq C \| g\|_{H^s}.
\]
In a similar way, we obtain the result for $\partial^2_t S(t),$ namely,
\begin{align*}
\| \partial^2_t S(t) g\|^2_{H^{s-1}}& \leq  \int_{\mathbb{R}^n}  \left| e^{-\frac{|\xi|^2 t}{2(1+|\xi|^2)}} \tfrac{|\xi|^2 (1+|\xi|^2)^{\frac{s-1}{2}}}{1+|\xi|^2} \cos \left(\tfrac{|\xi|^2 \sqrt{3+4|\xi|^2}t} {2(1+|\xi|^2)}\right)\widehat{g}(\xi)\right|^2 d\xi\\
& + \int_{\mathbb{R}^n}  \left| e^{-\frac{|\xi|^2 t}{2(1+|\xi|^2}}\tfrac{ |\xi|^2(1+2|\xi|^2)(1+|\xi|^2)^{\frac{s-1}{2}}}{ (1+|\xi|^2)\sqrt{3+4|\xi|^2}} \sin\left(\tfrac{|\xi|^2 \sqrt{3+4|\xi|^2}t} {2(1+|\xi|^2)}\right)\widehat{g}(\xi)\right|^2 d\xi\\
& \leq  \int_{\mathbb{R}^n}  \left| (1+|\xi|^2)^{\frac{s-1}{2}} \widehat{g}(\xi)\right|^2 d\xi + \int_{\mathbb{R}^n}  \left|  (1+|\xi|^2)^{\frac{s}{2}} \widehat{g}(\xi)\right|^2 d\xi\\
& \leq   \| g\|^2_{H^{s-1}}+ \| g\|^2_{H^s}\leq  \| g\|^2_{H^s}.
\end{align*}
Then, we conclude the third inequality for $k=1,$ that is,    
 \[
\|\partial^2_t S(t) g|_{H^s} \leq C \| g\|_{H^s}.
\]
For the operator $S(t)\Delta,$ we have
\begin{align*}
\|S(t)\Delta g\|^2_{H^s}& \leq  \int_{\mathbb{R}^n}  \left| e^{-\frac{|\xi|^2 t}{2(1+|\xi|^2)}}\tfrac{2(1+|\xi|^2) (1+|\xi|^2)^{\frac s2}}{\sqrt{3+4|\xi|^2}} \sin\left(\tfrac{|\xi|^2 \sqrt{3+4|\xi|^2}t} {2(1+|\xi|^2)}\right)\widehat{g}(\xi) \right|^2 d\xi\\
& \leq  C \int_{\mathbb{R}^n}  \left| (1+|\xi|^2)^{\frac {s+1}{2}}  \widehat{g}(\xi) \right|^2 d\xi \leq C   \| g\|^2_{H^{s+1}}.
\end{align*}
Therefore, we obtain the fourth inequality, with $k=0,$ namely,
\[
\|S(t)\Delta g\|_{H^s} \leq   C   \| g\|_{H^{s+1}}.
\]
Finally, for the case $\partial_t S(t)\Delta, $ reasoning similarly as in the proof of the last inequality, we arrive at
\begin{align*}
\|\partial_t S(t)\Delta g\|^2_{H^{s-1}}& \leq \int_{\mathbb{R}^n}  \left|  e^{-\frac{|\xi|^2 t}{2(1+|\xi|^2)}} |\xi|^2 (1+|\xi|^2)^{\frac {s-1}{2}} \cos \left(\tfrac{|\xi|^2 \sqrt{3+4|\xi|^2}t} {2(1+|\xi|^2)}\right)\widehat{g}(\xi) \right|^2 d\xi \\
& +\int_{\mathbb{R}^n}  \left|  e^{-\frac{|\xi|^2 t}{2(1+|\xi|^2}}\tfrac{|\xi|^2  (1+|\xi|^2)^{\frac {s-1}{2}}}{ \sqrt{3+4|\xi|^2}} \sin\left(\tfrac{|\xi|^2 \sqrt{3+4|\xi|^2}t} {2(1+|\xi|^2)}\right)\widehat{g}(\xi) \right|^2  d\xi\\
& \leq \int_{\mathbb{R}^n}  \left|  (1+|\xi|^2)^{\frac {s+1}{2}} \widehat{g}(\xi) \right|^2 d\xi +\int_{\mathbb{R}^n}  \left|   (1+|\xi|^2)^{\frac {s}{2}} \widehat{g}(\xi) \right|^2  d\xi\\
& \leq   \| g\|^2_{H^{s+1}}+ \| g\|^2_{H^s}\leq  \| g\|^2_{H^{s+1}}.
\end{align*}
Then, we arrived at the fourth inequality, for $k=1,$ that is,
\[
\|\partial_t S(t)\Delta g\|_{H^{s-1}}  \leq  C \| g\|_{H^{s+1}},
\]
which finishes the proof of the lemma.
\endproof

The proof of the next corollary follows from Lemma \ref{DatIni1} and arguing as in Lemma \ref{HsL1(1+t)<=1}. We omit the details. 
 \begin{corollary}\label{CoroDat}
Let  $s>\frac{n-2}{2}$, there exists a constant $C=C_n>0$ such that
\[\Vert \partial_t S(t)g \Vert_{L^{\infty}(\mathbb{R}^n)}\leq C (1+t)^{-\frac{n}{2}} \left[ \Vert g\Vert_{L^{1}(\mathbb{R}^n)}+ \Vert g\Vert_{H^{s+1}(\mathbb{R}^n)}\right],\]
\[\Vert S(t) \Delta g \Vert_{L^{\infty}(\mathbb{R}^n)}\leq C (1+t)^{-\frac{n}{2}} \left[\Vert g\Vert_{L^{1}(\mathbb{R}^n)}+\Vert g\Vert_{H^{s+2}(\mathbb{R}^n)}\right],\]
 for all $g\in \mathscr{S}(\mathbb{R}^n)$ and $t> 0.$
\end{corollary}
\subsection{Estimates in $H^s_p(\mathbb{R}^n)$}
Previous lemmas are essential to prove existence of global solutions in $H^s(\mathbb{R}^n)$ spaces. On the other hand, to prove the existence of global and local solutions in Bessel potential spaces $H^s_p(\mathbb{R}^n),$ we need to obtain time decay estimates for the solution of the linear problem (\ref{LinEqu}) and for the operator $\Lambda_{\theta}(t),$ which acts on the nonlinear part of the integro-differential equation (\ref{e3}).
\begin{lemma}\label{LamThe<=1}
Let   $\theta \in (\frac{2-n}{2},1]$ if $n=1,2$ and $\theta\in [0,1]$ if $n\geq 3.$ Assume $\sigma<3-n-2\theta$,  $2\leq p \leq  \infty$ and $\frac{1}{p}+\frac{1}{p^{\prime }}=1.$ There exists a constant $C=C_{\sigma,n}>0$ such that
\[\Vert \Lambda_{\theta} (t)g \Vert_{H_p^{\sigma}(\mathbb{R}^n)}\leq C \ t^{1-\theta -\frac n2 (1-\frac 2p)} \Vert g\Vert_{L^{p'}(\mathbb{R}^n)},\]
for all $g\in \mathscr{S}(\mathbb{R}^n)$ and $t> 0.$ Here $\Lambda_{\theta} (t)=S(t)(I-\Delta)^{-1}(-\Delta)^{\theta}.$ 
\end{lemma}
\proof
Consider $K_{\theta}(t)= \Lambda_{\theta} (t) J^{\sigma},$ where $J^\sigma=(I-\Delta)^{\frac{\sigma}{2}}$ is the Bessel potential operator. By the boundedness of the $\sin(\cdot)$ function and since $  \tfrac{(1+|\xi|^2)^{\frac{\sigma}{2}} }{ \sqrt{3+4|\xi|^2}} \apprle1,$ we conclude that
\begin{align*}
|K_{\theta}(t)g(x)|& \leq  \left| \int_{|\xi|\leq 1}e^{-\frac{|\xi|^2 t}{2(1+|\xi|^2)}}  \tfrac{2(1+|\xi|^2)^{\frac{\sigma}{2}} |\xi|^{2\theta}}{|\xi|^2 \sqrt{3+4|\xi|^2}}  \sin\left(\tfrac{|\xi|^2 \sqrt{3+4|\xi|^2}t} {2(1+|\xi|^2)}\right)\widehat{g}(\xi)e^{ix\cdot\xi} d\xi\right|\\
& +\left| \int_{|\xi|\geq 1}e^{-\frac{|\xi|^2 t}{2(1+|\xi|^2)}}   \tfrac{2(1+|\xi|^2)^{\frac{\sigma}{2}} |\xi|^{2\theta}}{ |\xi|^2 \sqrt{3+4|\xi|^2}} \sin\left(\tfrac{|\xi|^2 \sqrt{3+4|\xi|^2}t} {2(1+|\xi|^2)}\right)\widehat{g}(\xi)e^{ix\cdot\xi} d\xi\right|\\
& \apprle\| \widehat{g}\|_{L^{\infty}}   \int_{|\xi|\leq 1}e^{-\frac{|\xi|^2 t}{2(1+|\xi|^2)}} |\xi|^{2(\theta-1)} d\xi + \|\widehat{g}\|_{L^{\infty}}  \int_{|\xi|\geq 1}e^{-\frac{|\xi|^2 t}{2(1+|\xi|^2)}}   \tfrac{2(1+|\xi|^2)^{\frac{\sigma}{2}} |\xi|^{2\theta}}{ |\xi|^2 \sqrt{3+4|\xi|^2}} d\xi.
\end{align*}
Since  $ \| \widehat{g}\|_{L^{\infty}} \leq    \|g\|_{L^1(\mathbb{R}^n)},$  $e^{-\frac{|\xi|^2 t}{2(1+|\xi|^2)}} \leq e^{- \frac{|\xi|^2t}{4}}$ for $|\xi| \leq 1$ and, $e^{-\frac{|\xi|^2 t}{2(1+|\xi|^2)}} \leq e^{- \frac t4}$ for $|\xi| \geq 1,$  we obtain
\begin{align*}
|K_{\theta}(t)g(x)| & \apprle   \|g\|_{L^1(\mathbb{R}^n)} \int_{|\xi|\leq 1}e^{-\frac{|\xi|^2 t}{4}} |\xi|^{2(\theta-1)} d\xi + e^{-\frac {t}{4}}  \|g\|_{L^1(\mathbb{R}^n)}  \int_{|\xi|\geq 1}  \tfrac{2(1+|\xi|^2)^{\frac{\sigma}{2}} |\xi|^{2\theta}}{ |\xi|^2 \sqrt{3+4|\xi|^2}} d\xi.
\end{align*}
Now, from Lemma \ref{AcotGamma}, spherical coordinates and the assumption on $\sigma$, we have
\begin{align*}
|K_{\theta}(t)g(x)|  & \apprle   t^{-\frac{n+2(\theta-1)}{2}}  \|g\|_{L^1(\mathbb{R}^n)}  + e^{-\frac {t}{4}}  \|g\|_{L^1(\mathbb{R}^n)}  \int_{|\xi|\geq 1}  |\xi|^{\sigma+2\theta-3} d\xi. \\
& \apprle   t^{-\frac{n+2(\theta-1)}{2}}  \|g\|_{L^1(\mathbb{R}^n)}  + e^{-\frac {t}{4}}  \|g\|_{L^1(\mathbb{R}^n)}  \int_1^{\infty}  r^{\sigma+2\theta-3} r^{n-1}dr\\
& \apprle  t^{-\frac{n+2(\theta-1)}{2}}   \|g\|_{L^1(\mathbb{R}^n)} + e^{-\frac {t}{4}} \|g\|_{L^1(\mathbb{R}^n)}.
\end{align*}
Using the last inequality and the fact that there exists a constant $C_{n,\theta}>0,$ that only depends on $n$ and $\theta,$ such that $e^{-\frac t4} \leq C_{n,\theta} t^{-\frac {n+2(\theta-1)}{2}},$ we conclude that 
\begin{equation}\label{KL1}
\|K_{\theta}(t)g\|_{L^{\infty}}   \apprle  t^{-\frac{n+2(\theta-1)}{2}}   \|g\|_{L^1(\mathbb{R}^n)} .
\end{equation}
Now, we prove that  $K_{\theta}(t):L^2(\mathbb{R}^n)\longrightarrow L^2(\mathbb{R}^n)$ is continuous. We consider two cases:
\begin{itemize}
\item Case 1: Assume that $t>1.$  Since $\left|  \tfrac{2(1+|\xi|^2)}{t|\xi|^2 \sqrt{3+4|\xi|^2}}  \sin\left(\tfrac{|\xi|^2 \sqrt{3+4|\xi|^2}t} {2(1+|\xi|^2)}\right)\right| \apprle 1$ for all $t>0$ and $\xi \in \mathbb{R}^n,$  we have
\begin{align*}
\|K_{\theta}(t)g\|_{L^2(\mathbb{R}^n)}^2& \leq  t^2 \int_{\mathbb{R}^n}  e^{-\frac{|\xi|^2 t}{1+|\xi|^2}} \left|  \tfrac{2 (1+|\xi|^2) (1+|\xi|^2)^{\frac{\sigma-2}{2}} |\xi|^{2\theta}} {t|\xi|^2 \sqrt{3+4|\xi|^2}}  
\sin\left(\tfrac{|\xi|^2 \sqrt{3+4|\xi|^2}t} {2(1+|\xi|^2)}\right)\right|^2  |\widehat{g}(\xi)|^2d\xi \\
& \leq  t^2 \int_{\mathbb{R}^n}e^{-\frac{|\xi|^2 t}{1+|\xi|^2}}  (1+|\xi|^2)^{\sigma-2} |\xi|^{4\theta} |\widehat{g}(\xi)|^2d\xi=I_1+I_2,
\end{align*}
where 
\[ I_1 =  t^2 \int_{|\xi|\leq 1}e^{-\frac{|\xi|^2 t}{1+|\xi|^2}}  (1+|\xi|^2)^{\sigma-2} |\xi|^{4\theta} |\widehat{g}(\xi)|^2d\xi \]
and 
\[ I_2=  t^2 \int_{|\xi|\geq 1}e^{-\frac{|\xi|^2 t}{1+|\xi|^2}}  (1+|\xi|^2)^{\sigma-2} |\xi|^{4\theta} |\widehat{g}(\xi)|^2d\xi. \]
Since  $e^{-\frac{|\xi|^2 t}{1+|\xi|^2}} \leq e^{-\frac{|\xi|^2 t}{2}}$  and $(1+|\xi|^2)^{\sigma-2}\leq \max\{2^{\sigma-2},1\}$ for $|\xi|\leq 1,$ we obtain  
\[ I_1 \leq  C_{\sigma} t^2 \int_{|\xi|\leq 1}e^{-\frac{|\xi|^2 t}{2}}  |\xi|^{4\theta} |\widehat{g}(\xi)|^2d\xi.\]
Now, note that $e^{-\frac{|\xi|^2 t}{2}}  |\xi|^{4\theta}=4^{\theta} t^{-2\theta} e^{-\left| \frac{\xi \sqrt{t}}{\sqrt{2}}\right|^2}  \left| \frac{\xi \sqrt{t}}{\sqrt{2}}\right|^{4\theta}\leq 4^{\theta} t^{-2\theta}.$ Then, we can conclude that
\begin{equation}\label{I_1impor}
I_1\leq C_{\sigma,\theta} \ t^{2(1-\theta)}   \|g\|^2_{L^2(\mathbb{R}^n)}.
\end{equation}
Next, since $e^{-\frac{|\xi|^2 t}{1+|\xi|^2}} \leq e^{-\frac{t}{2}}$  for $|\xi|\geq 1,$ we get
\begin{equation}\label{I_2}
I_2 \leq  t^2 e^{-\frac t2} \int_{|\xi|\geq 1}  (1+|\xi|^2)^{\sigma-2} |\xi|^{4\theta}  |\widehat{g}(\xi)|^2d\xi.
\end{equation}
The restriction $\sigma<3-n-2\theta$, implies $\sigma<2(1-\theta)$ and therefore $  (1+|\xi|^2)^{\sigma-2} |\xi|^{4\theta} \leq 1.$ From (\ref{I_2}) we arrive at
\[
I_2 \leq  t^2 e^{-\frac t2}    \|g\|^2_{L^2(\mathbb{R}^n)}.
\]
Using the fact that $t>1,$ there exists a constant $C_{\theta}>0$ such that $e^{-\frac t2}\leq C_{\theta} t^{-2\theta}.$ Therefore,
\begin{equation}\label{I_2impor}
I_1\leq C_{\sigma,\theta} \ t^{2(1-\theta)}   \|g\|^2_{L^2(\mathbb{R}^n)}.
\end{equation}
Combining (\ref{I_1impor}) and (\ref{I_2impor}) and taking square root, we have
\begin{equation*}
\|K_{\theta}(t)g\|_{L^2(\mathbb{R}^n)}  \apprle t^{1-\theta}  \|g\|_{L^2(\mathbb{R}^n)}. 
\end{equation*}
\item Case 2: Consider $0<t\leq 1.$ Since $\left|  \tfrac{2(1+|\xi|^2)}{t|\xi|^2 \sqrt{3+4|\xi|^2}}  \sin\left(\tfrac{|\xi|^2 \sqrt{3+4|\xi|^2}t} {2(1+|\xi|^2)}\right)\right| \apprle 1, $  
$(1+|\xi|^2)^{\sigma-2} |\xi|^{4\theta} \leq 1$ and $e^{-\frac{|\xi|^2 t}{1+|\xi|^2}} \leq 1$   for all $\xi \in \mathbb{R}^n,$ we obtain
\begin{align*}
\|K_{\theta}(t)g\|_{L^2(\mathbb{R}^n)}^2& \leq  t^2 \int_{\mathbb{R}^n}  e^{-\frac{|\xi|^2 t}{1+|\xi|^2}} \left|  \tfrac{2 (1+|\xi|^2) (1+|\xi|^2)^{\frac{\sigma-2}{2}} |\xi|^{2\theta}} {t|\xi|^2 \sqrt{3+4|\xi|^2}}  
\sin\left(\tfrac{|\xi|^2 \sqrt{3+4|\xi|^2}t} {2(1+|\xi|^2)}\right)\right|^2  |\widehat{g}(\xi)|^2d\xi \\
& \leq  t^2 \int_{\mathbb{R}^n} |\widehat{g}(\xi)|^2d\xi=t^{2(1-\theta)}t^{2\theta} \|g\|^2_{L^2(\mathbb{R}^n)}.
\end{align*}
Since $0<t\leq 1$ and $\theta \geq 0,$ we have $t^{2\theta}\leq 1.$ Then, from last inequality, we arrive at 
\begin{equation}\label{KL2}
\|K_{\theta}(t)g\|_{L^2(\mathbb{R}^n)}  \apprle t^{1-\theta}  \|g\|_{L^2(\mathbb{R}^n)}. 
\end{equation}
In any case, we have that $K_{\theta}(t)$ is a bounded operator from $L^2(\mathbb{R}^n)$ to $L^2(\mathbb{R}^n).$  
\end{itemize}
Therefore, recalling that $K_{\theta}(t)= \Lambda_{\theta}(t)J^\sigma,$ from (\ref{KL1}), (\ref{KL2}) and applying the Riesz-Thorin interpolation theorem, we conclude the proof of the lemma. 
\endproof

The next corollary will be useful to estimate the nonlinear part of equation (\ref{e3}) (see Proposition \ref{nl2g}).
\begin{corollary}\label{Coro1}
Let   $s\in \mathbb{R},$  $\theta \in (\frac{2-n}{2},1]$ if $n=1,2$ and $\theta\in [0,1]$ if $n\geq 3.$  There exists $C=C_{n}>0$ such that
\[ \Vert \Lambda_{\theta} (t) g  \Vert_{H^s(\mathbb{R}^n)}\leq C \ t^{1-\theta} \Vert g\Vert_{H^s(\mathbb{R}^n)},\]
for all $g\in \mathscr{S}(\mathbb{R}^n)$ and $t> 0.$ 
\end{corollary}
\proof
Repeating the second part of the proof of Lemma \ref{LamThe<=1}, with $\sigma=0,$ we obtain (\ref{I_2}), that is,
\begin{equation*}
I_2 \leq  t^2 e^{-\frac t2} \int_{|\xi|\leq 1}  (1+|\xi|^2)^{-2} |\xi|^{4\theta}  |\widehat{g}(\xi)|^2d\xi.
\end{equation*}
Since $\frac{2-n}{2} <\theta \leq 1,$ we have $ (1+|\xi|^2)^{-2} |\xi|^{4\theta} \leq 1.$ Therefore  $I_2 \leq C_{\theta} \ t^{2(1-\theta)}   \|g\|^2_{L^2(\mathbb{R}^n)}.$ Following the rest of the proof of Lemma  \ref{LamThe<=1}, we arrive at 
\[
\|\Lambda_{\theta}(t)g\|_{L^2(\mathbb{R}^n)}  \apprle t^{1-\theta}  \|g\|_{L^2(\mathbb{R}^n)}. 
\]
From the last inequality, we have
\[
\|\Lambda_{\theta}(t)g\|_{H^s(\mathbb{R}^n)}  =\|\Lambda_{\theta}(t) J^sg\|_{L^2(\mathbb{R}^n)}  \apprle t^{1-\theta}  \|J^s g\|_{L^2(\mathbb{R}^n)}=C_n t^{1-\theta}  \|g\|_{H^s(\mathbb{R}^n)}, 
\]
which finishes the proof of the corollary.
\endproof
\begin{lemma}\label{DtLamThe<=1}
Let   $0\leq \theta \leq 1,$ $\sigma<3-n-2\theta$,  $2\leq p \leq  \infty$ and $\frac{1}{p}+\frac{1}{p^{\prime }}=1.$ There exists $C=C_{\sigma, n}>0$ such that
\[\Vert \partial_t \Lambda_{\theta} (t)g \Vert_{H_p^{\sigma-1}(\mathbb{R}^n)}\leq C \ t^{-\frac n2(1-\frac 2p)} \Vert g\Vert_{L^{p'}(\mathbb{R}^n)},\]
for all $g\in \mathscr{S}(\mathbb{R}^n)$ and $t> 0.$ Here $\Lambda_{\theta}(t)=S(t)(I-\Delta)^{-1}(-\Delta)^{\theta}.$   
\end{lemma}
\proof
Consider $M_{\theta}(t)=\partial_t \Lambda_{\theta} (t) J^{\sigma-1},$ where $J^{s}=(I-\Delta)^{s}$ is the Bessel potential operator.  Similarly as the proof of Lemma \ref{LamThe<=1}, we have
\begin{align*}
|M_{\theta}(t)g(x)|& \leq \left| \int_{|\xi|\leq 1}e^{-\frac{|\xi|^2 t}{2(1+|\xi|^2)}} \frac{|\xi|^{2\theta}(1+|\xi|^2)^{\frac{\sigma-1}{2}}}{1+|\xi|^2} \cos \left(\tfrac{|\xi|^2 \sqrt{3+4|\xi|^2}t} {2(1+|\xi|^2)}\right)\widehat{g}(\xi)e^{ix\cdot\xi} d\xi\right|\\
&+\left| \int_{|\xi|\geq 1}e^{-\frac{|\xi|^2 t}{2(1+|\xi|^2)}} \frac{|\xi|^{2\theta}(1+|\xi|^2)^{\frac{\sigma-1}{2}}}{1+|\xi|^2} \cos \left(\tfrac{|\xi|^2 \sqrt{3+4|\xi|^2}t} {2(1+|\xi|^2)}\right)\widehat{g}(\xi)e^{ix\cdot\xi} d\xi\right|\\
& +\left| \int_{|\xi|\leq  1}e^{-\frac{|\xi|^2 t}{2(1+|\xi|^2}}  \tfrac{|\xi|^{2\theta}(1+|\xi|^2)^{\frac{\sigma-1}{2}}}{(1+|\xi|^2)\sqrt{3+4|\xi|^2}}  \sin\left(\tfrac{|\xi|^2 \sqrt{3+4|\xi|^2}t} {2(1+|\xi|^2)}\right)\widehat{g}(\xi)e^{ix\cdot\xi} d\xi\right|\\
&+\left| \int_{|\xi|\geq 1}e^{-\frac{|\xi|^2 t}{2(1+|\xi|^2)}} \tfrac{|\xi|^{2\theta}(1+|\xi|^2)^{\frac{\sigma-1}{2}}}{(1+|\xi|^2)\sqrt{3+4|\xi|^2}} \sin\left(\tfrac{|\xi|^2 \sqrt{3+4|\xi|^2}t} {2(1+|\xi|^2)}\right)\widehat{g}(\xi)e^{ix\cdot\xi} d\xi\right|\\
& \apprle t^{-\frac n2} \|g\|_{L^1(\mathbb{R}^n)}  + e^{-\frac{t}{4}} \|g\|_{L^1(\mathbb{R}^n)} \int_{|\xi|\geq 1} |\xi|^{\sigma-3+2\theta} d\xi + e^{-\frac{t}{4}} \|g\|_{L^1(\mathbb{R}^n)} \int_{|\xi|\geq 1} |\xi|^{\sigma-4+2\theta} d\xi\\
& \apprle t^{-\frac n2} \|g\|_{L^1(\mathbb{R}^n)}  + e^{-\frac{t}{4}} \|g\|_{L^1(\mathbb{R}^n)} \int_{1}^{\infty} r^{\sigma+2\theta+n-4} dr + e^{-\frac{t}{4}} \|g\|_{L^1(\mathbb{R}^n)} \int_{1}^{\infty} r^{\sigma+2\theta+n-5} dr \\
& \apprle (t^{-\frac n2} +e^{-\frac{t}{4}})\|g\|_{L^1(\mathbb{R}^n)}  \apprle t^{-\frac n2} \|g\|_{L^1(\mathbb{R}^n)}.
\end{align*}
Therefore
\begin{equation}\label{KTheta1}
\|M_{\theta}(t)g\|_{L^{\infty}(\mathbb{R}^n)}  \apprle t^{-\frac n2} \|g\|_{L^1(\mathbb{R}^n)}.
\end{equation}
Now, we prove that  $M_{\theta}(t):L^2(\mathbb{R}^n)\longrightarrow L^2(\mathbb{R}^n)$ is continuous. Indeed, since $\sigma-1+2(\theta-1)<-n\leq -1,$ we obtain
$\tfrac{|\xi|^{2\theta}(1+|\xi|^2)^{\frac{\sigma-1}{2}}}{1+|\xi|^2} \apprle 1$ and $\tfrac{|\xi|^{2\theta}(1+|\xi|^2)^{\frac{\sigma-1}{2}}}{(1+|\xi|^2)\sqrt{3+4|\xi|^2}}\apprle 1.$ From the boundedness of the functions $\cos(\cdot),$ $\sin(\cdot),$ $e^{-\frac{|\xi|^2 t}{1+|\xi|^2}}$, we conclude that 
\begin{align*}
\|M_{\theta}(t)g\|_{L^2(\mathbb{R}^n)}^2& \leq \int_{\mathbb{R}^n}e^{-\frac{|\xi|^2 t}{1+|\xi|^2}} \left| \frac{|\xi|^{2\theta}(1+|\xi|^2)^{\frac{\sigma-1}{2}}}{1+|\xi|^2} \cos \left(\tfrac{|\xi|^2 \sqrt{3+4|\xi|^2}t} {2(1+|\xi|^2)}\right) \right|^2  |\widehat{g}(\xi)|^2d\xi \\
& +\int_{\mathbb{R}^n}e^{-\frac{|\xi|^2 t}{1+|\xi|^2}} \left| \tfrac{|\xi|^{2\theta}(1+|\xi|^2)^{\frac{\sigma-1}{2}}}{(1+|\xi|^2)\sqrt{3+4|\xi|^2}} \sin \left(\tfrac{|\xi|^2 \sqrt{3+4|\xi|^2}t} {2(1+|\xi|^2)}\right) \right|^2  |\widehat{g}(\xi)|^2d\xi \\ 
&\leq C \|g\|^2_{L^2(\mathbb{R}^n)} . 
\end{align*}
Taking square root, we have
\begin{equation}\label{KTheta2}
\|M_{\theta}(t)g\|_{L^2(\mathbb{R}^n)}  \apprle \|g\|_{L^2(\mathbb{R}^n)}. 
\end{equation}
Then, recalling that $M_{\theta}(t)= \partial_t \Lambda(t)J^{\sigma-1},$ from (\ref{KTheta1}) and (\ref{KTheta2})  applying the Riesz-Thorin interpolation theorem, we conclude the proof of the lemma. 
\endproof

As before, the next corollary will be useful to estimate the nonlinear part of equation (\ref{e3}) (see Proposition \ref{nl2g}).
\begin{corollary}\label{Coro2}
Let   $0 \leq \theta \leq 1$ and $s\in \mathbb{R}.$ There exists $C=C_{n}>0$ such that
\[ \Vert \partial_t \Lambda_{\theta} (t) g  \Vert_{H^{s-1}(\mathbb{R}^n)}\leq C \ \Vert g\Vert_{H^{s-1}(\mathbb{R}^n)},\]
for all $g\in \mathscr{S}(\mathbb{R}^n)$ and $t> 0.$ 
\end{corollary}
\proof
Repeating the second part of the proof of Lemma \ref{DtLamThe<=1}, with $\sigma=1,$ we obtain
\begin{align*}
\|\partial_t \Lambda_{\theta}(t)g\|_{L^2(\mathbb{R}^n)}^2& \leq \int_{\mathbb{R}^n}e^{-\frac{|\xi|^2 t}{1+|\xi|^2}} \left| \frac{|\xi|^{2\theta}}{1+|\xi|^2} \cos \left(\tfrac{|\xi|^2 \sqrt{3+4|\xi|^2}t} {2(1+|\xi|^2)}\right) \right|^2  |\widehat{g}(\xi)|^2d\xi \\
& +\int_{\mathbb{R}^n}e^{-\frac{|\xi|^2 t}{1+|\xi|^2}} \left| \tfrac{|\xi|^{2\theta}}{(1+|\xi|^2)\sqrt{3+4|\xi|^2}} \sin \left(\tfrac{|\xi|^2 \sqrt{3+4|\xi|^2}t} {2(1+|\xi|^2)}\right) \right|^2  |\widehat{g}(\xi)|^2d\xi.
\end{align*}
Since $0\leq \theta \leq 1,$ we have $ \frac{|\xi|^{2\theta}}{1+|\xi|^2} \leq 1$ and $\tfrac{|\xi|^{2\theta}}{(1+|\xi|^2)\sqrt{3+4|\xi|^2}} \leq 1.$ Using the facts that $e^{-\frac{|\xi|^2 t}{1+|\xi|^2}} \leq 1$ and since $\cos(\cdot), \sin(\cdot) $ are bounded functions, we arrive at 
\[
\|\partial_t \Lambda_{\theta}(t)g\|_{L^2(\mathbb{R}^n)}  \leq C \|g\|_{L^2(\mathbb{R}^n)}. 
\]
From the last inequality, we have
\[
\|\partial_t \Lambda_{\theta}(t)g\|_{H^{s-1}(\mathbb{R}^n)}  =\|\partial_t \Lambda_{\theta}(t) J^{s-1}g\|_{L^2(\mathbb{R}^n)}  \apprle   \|J^{s-1} g\|_{L^2(\mathbb{R}^n)}=C_n   \|g\|_{H^{s-1}(\mathbb{R}^n)}, 
\]
which finishes the proof of the corollary.
\endproof

Next lemma is useful to bound the linear part in the proof of Theorem \ref{teo_local_1}.
\begin{lemma}\label{DatIni2}
Let   $\sigma<1-n$, $2\leq p \leq  \infty$ and $\frac{1}{p}+\frac{1}{p^{\prime }}=1.$ There exists a constant $C=C_{\sigma}>0,$ such that
\begin{eqnarray}
\Vert \partial^k_t S(t)g \Vert_{H_p^{\sigma-k}(\mathbb{R}^n)}\leq C t^{-\frac n 2(1-\frac 2p)}\Vert g \Vert_{L^{p'}(\mathbb{R}^n)},\label{p1}\\
\Vert \partial^{k-1}_t S(t) \Delta g \Vert_{H_p^{\sigma-k+1}(\mathbb{R}^n)}\leq C t^{-\frac n 2(1-\frac 2p)}\Vert g \Vert_{H^2_{p'}(\mathbb{R}^n)},\label{p2} 
\end{eqnarray}
for all $g\in \mathscr{S}(\mathbb{R}^n),$ $k=1,2$ and $t> 0.$  
\end{lemma}
\proof
Consider $K_1(t)=\partial_t S(t) J^{\sigma-1}.$ Similarly as the proof of Lemma \ref{DtLamThe<=1}, we obtain
\begin{align*}
|K_1(t)g(x)|& \leq \left| \int_{|\xi|\leq 1} e^{-\frac{|\xi|^2 t}{2(1+|\xi|^2)}} (1+|\xi|^2)^{\frac{\sigma-1}{2}} \cos \left(\tfrac{|\xi|^2 \sqrt{3+4|\xi|^2}t} {2(1+|\xi|^2)}\right)\widehat{g}(\xi)e^{ix\cdot\xi} d\xi\right|\\
& +\left| \int_{|\xi|\leq 1} e^{-\frac{|\xi|^2 t}{2(1+|\xi|^2}}\tfrac{ (1+|\xi|^2)^{\frac{\sigma-1}{2}}}{ \sqrt{3+4|\xi|^2}} \sin\left(\tfrac{|\xi|^2 \sqrt{3+4|\xi|^2}t} {2(1+|\xi|^2)}\right)\widehat{g}(\xi)e^{ix\cdot\xi} d\xi\right|\\
&+\left| \int_{|\xi|\geq 1} e^{-\frac{|\xi|^2 t}{2(1+|\xi|^2)}} (1+|\xi|^2)^{\frac{\sigma-1}{2}} \cos \left(\tfrac{|\xi|^2 \sqrt{3+4|\xi|^2}t} {2(1+|\xi|^2)}\right)\widehat{g}(\xi)e^{ix\cdot\xi} d\xi\right|\\
& +\left| \int_{|\xi|\geq 1} e^{-\frac{|\xi|^2 t}{2(1+|\xi|^2}}\tfrac{ (1+|\xi|^2)^{\frac{\sigma-1}{2}}}{ \sqrt{3+4|\xi|^2}} \sin\left(\tfrac{|\xi|^2 \sqrt{3+4|\xi|^2}t} {2(1+|\xi|^2)}\right)\widehat{g}(\xi)e^{ix\cdot\xi} d\xi\right|\\
& \apprle t^{-\frac n2}  \|g\|_{L^1(\mathbb{R}^n)} + e^{-\frac t4}\|g\|_{L^1(\mathbb{R}^n)} \int_{|\xi|\geq 1} (1+|\xi|^2)^{\frac{\sigma-1}{2}} d\xi+ e^{-\frac t4}\|g\|_{L^1(\mathbb{R}^n)} \int_{|\xi|\geq 1} (1+|\xi|^2)^{\frac{\sigma-2}{2}} d\xi \\
& \apprle  t^{-\frac n2}  \|g\|_{L^1(\mathbb{R}^n)} + e^{-\frac t4}\|g\|_{L^1(\mathbb{R}^n)}  \int_{1}^{\infty} r^{\sigma-1}r^{n-1} dr  +  e^{-\frac t4}\|g\|_{L^1(\mathbb{R}^n)}  \int_{1}^{\infty} r^{\sigma-2}r^{n-1} dr  \\
& \apprle    t^{-\frac n2}  \|g\|_{L^1(\mathbb{R}^n)} + e^{-\frac t4}\|g\|_{L^1(\mathbb{R}^n)}  \apprle    t^{-\frac n2}  \|g\|_{L^1(\mathbb{R}^n)}.
\end{align*}
Now, is not difficult to see that 
\[
\|K_1(t)g\|_{L^2(\mathbb{R}^n)}  \apprle \|g\|_{L^2(\mathbb{R}^n)}. 
\]
Then, recalling that $K_1(t)= \partial_t S(t)J^{\sigma},$ and applying the Riesz-Thorin interpolation theorem, we conclude the proof of (\ref{p1}), for $k=1$. \\

On the other hand, considering $K_2(t)=\partial^2_t S(t) J^{\sigma-2},$ we get  
\begin{align*}
|K_2(t)g(x)|& \apprle \left| \int_{|\xi|\leq 1} e^{-\frac{|\xi|^2 t}{2(1+|\xi|^2)}} \tfrac{|\xi|^2 }{1+|\xi|^2} (1+|\xi|^2)^{\frac{\sigma-2}{2}} \cos \left(\tfrac{|\xi|^2 \sqrt{3+4|\xi|^2}t} {2(1+|\xi|^2)}\right)\widehat{g}(\xi)e^{ix\cdot\xi} d\xi\right|\\
& +\left| \int_{|\xi|\leq 1} e^{-\frac{|\xi|^2 t}{2(1+|\xi|^2}}\tfrac{ |\xi|^2(1+2|\xi|^2)(1+|\xi|^2)^{\frac{\sigma-2}{2}}}{ (1+|\xi|^2)\sqrt{3+4|\xi|^2}} \sin\left(\tfrac{|\xi|^2 \sqrt{3+4|\xi|^2}t} {2(1+|\xi|^2)}\right)\widehat{g}(\xi)e^{ix\cdot\xi} d\xi\right|\\
& +  \left| \int_{|\xi|\geq 1} e^{-\frac{|\xi|^2 t}{2(1+|\xi|^2)}} \tfrac{|\xi|^2 }{1+|\xi|^2} (1+|\xi|^2)^{\frac{\sigma-2}{2}} \cos \left(\tfrac{|\xi|^2 \sqrt{3+4|\xi|^2}t} {2(1+|\xi|^2)}\right)\widehat{g}(\xi)e^{ix\cdot\xi} d\xi\right|\\
& +\left| \int_{|\xi|\geq 1} e^{-\frac{|\xi|^2 t}{2(1+|\xi|^2}}\tfrac{ |\xi|^2(1+2|\xi|^2)(1+|\xi|^2)^{\frac{\sigma-2}{2}}}{ (1+|\xi|^2)\sqrt{3+4|\xi|^2}} \sin\left(\tfrac{|\xi|^2 \sqrt{3+4|\xi|^2}t} {2(1+|\xi|^2)}\right)\widehat{g}(\xi)e^{ix\cdot\xi} d\xi\right|\\
& \apprle t^{-\frac n2} \|g\|_{L^1(\mathbb{R}^n)} + e^{-\frac t4} \|g\|_{L^1(\mathbb{R}^n)} \int_{|\xi|\geq 1} (1+|\xi|^2)^{\frac{\sigma-2}{2}} d\xi  + e^{-\frac t4}  \|g\|_{L^1(\mathbb{R}^n)} \int_{|\xi|\geq 1} (1+|\xi|^2)^{\frac{\sigma-1}{2}} d\xi\\
& \apprle t^{-\frac n2} \|g\|_{L^1(\mathbb{R}^n)} + e^{-\frac t4} \|g\|_{L^1(\mathbb{R}^n)}  \int_{0}^{\infty} r^{\sigma-2}r^{n-1} dr  + e^{-\frac t4}  \|g\|_{L^1(\mathbb{R}^n)}  \int_{0}^{\infty} r^{\sigma-1}r^{n-1} dr   \\
& \apprle  t^{-\frac n2} \|g\|_{L^1(\mathbb{R}^n)}   + e^{-\frac t4}  \|g\|_{L^1(\mathbb{R}^n)}    \apprle  t^{-\frac n2} \|g\|_{L^1(\mathbb{R}^n)}.
\end{align*}
Also, as before
\[
\|K_2(t)g\|_{L^2(\mathbb{R}^n)}  \apprle \|g\|_{L^2(\mathbb{R}^n)}. 
\]
Then, recalling that $K_2(t)= \partial^2_t S(t)J^{\sigma-2},$ and applying the Riesz-Thorin interpolation theorem, we conclude the proof of (\ref{p1}), for $k=2$. \\
Now, we prove inequality (\ref{p2}). From Lemma \ref{LamThe<=1} with $\theta=1$, we have
\[
\Vert S(t) \Delta g \Vert_{H_p^{\sigma}(\mathbb{R}^n)} = \Vert \Lambda_1 (t) (I-\Delta) g \Vert_{H_p^{\sigma}(\mathbb{R}^n)}  \leq C t^{-\frac n 2(1-\frac 2p)}\Vert (I-\Delta)g \Vert_{L^{p'}(\mathbb{R}^n)}\leq  C t^{-\frac n 2(1-\frac 2p)}\Vert g \Vert_{H^2_{p'}(\mathbb{R}^n)},
\]
which finishes the proof of the inequality (\ref{p2}) with $k=1.$ In a similar way as before, by Lemma \ref{DtLamThe<=1} with $\theta=1$, we obtain 
\[
\Vert \partial_t S(t) \Delta g \Vert_{H_p^{\sigma-1}(\mathbb{R}^n)} = \Vert \partial_t \Lambda_1 (t) (I-\Delta) g \Vert_{H_p^{\sigma-1}(\mathbb{R}^n)}  \leq C t^{-\frac n 2(1-\frac 2p)}\Vert (I-\Delta)g \Vert_{L^{p'}(\mathbb{R}^n)}\leq  C t^{-\frac n 2(1-\frac 2p)}\Vert g \Vert_{H^2_{p'}(\mathbb{R}^n)},
\]
which finishes the proof of the Lemma.
\endproof

The next lemmas are useful to deal with the nonlinear part of the integro-differential equation (\ref{e3}).
\begin{lemma}[Banquet-Ferreira-Villamizar \protect\cite{Banquet1}]
\label{NonIne1} Let $\lambda \geq 2,$ $1<p,q<\infty $ and $s>0$ be such that $s<\min \{\frac{n}{q},\lambda -1\}$ and
\begin{equation*}
1-\frac{1}{p}\geq \frac{1}{q}+\frac{\lambda -1}{n}\left( \frac{n}{q}%
-s\right) .
\end{equation*}
There exists an universal constant $C>0$ such that
\begin{equation}
\left\Vert |f|^{\lambda }-|g|^{\lambda }\right\Vert _{H_{p^{\prime
}}^{s}}\leq C\Vert f-g\Vert _{H_{q}^{s}}\left[ \Vert f\Vert
_{H_{q}^{s}}^{\lambda -1}+\Vert g\Vert _{H_{q}^{s}}^{\lambda -1}\right] .
\label{aux-inl}
\end{equation}%
Furthermore, if $\lambda $ is a positive odd integer then (\ref{aux-inl})
holds without the restriction $s<\lambda -1.$
\end{lemma}
\begin{lemma}[Zhu \protect\cite{Zhu}, Lemma 3.4]\label{leibnitz}
Let $\lambda\geq 2$ be a positive integer. If $u, \tilde{u}\in H^s\cap L^\infty$ and $\Vert u\Vert_{L^\infty}\leq M,$ $\Vert \tilde{u}\Vert_{L^\infty}\leq M,$ then 
\begin{eqnarray*}
\Vert \vert u\vert^\lambda-\vert \tilde{u}\vert^\lambda\Vert_{H^s}&\leq & C(M) [\Vert u-\tilde{u}\Vert_{L^\infty}(\Vert u\Vert_{H^s}+\Vert \tilde{u}\Vert_{H^s}) (\Vert u\Vert_{L^\infty}+\Vert \tilde{u}\Vert_{L^\infty})^{\lambda-2}\nonumber\\
&&+\Vert u-\tilde{u}\Vert_{H^s}(\Vert u\Vert_{L^\infty}+\Vert \tilde{u}\Vert_{L^\infty})^{\lambda-1}],\\
\Vert \vert u\vert^\lambda-\vert \tilde{u}\vert^\lambda\Vert_{L^1}&\leq & C(M) [(\Vert u\Vert_{L^\infty}+\Vert \tilde{u}\Vert_{L^\infty})^{\lambda-2}(\Vert u\Vert_{L^2}+\Vert \tilde{u}\Vert_{L^2})\Vert u-\tilde{u}\Vert_{L^2}],
\end{eqnarray*}
where $C(M)$ is a constant dependent on $M.$
\end{lemma}
We finish this section by recalling the following technical lemma.
\begin{lemma}[Zhu  \protect\cite{Zhu}, Lemma 3.5] \label{lem_int}
Let $a,b$ non negative real constants such that $b\geq a\geq 0.$ Then,
\begin{eqnarray*}
\int_0^t(1+t-\tau)^{-a}(1+\tau)^{-b}d\tau\leq C(1+t)^{-a}\int_0^t(1+\tau)^{-b}d\tau.
\end{eqnarray*}
\end{lemma}
\section{Global solutions}
Before proving the existence of global solutions, we establish an estimate of the nonlinear term of integro-differential equation (\ref{e3}) in the norm
\begin{eqnarray*}
\Vert u\Vert_{\mathcal{Y}^{s,\alpha_1}}&:=&\sup_{0<t<\infty}((1+t)^{\alpha_1}\Vert u(t)\Vert_{L^\infty}+ \Vert u(t)\Vert_{H^s}+\Vert u_t(t)\Vert_{H^{s-1}}),\ \alpha_1=\tfrac{n+2(\theta-1)}{2},
\end{eqnarray*}
as well as in the norm
\begin{eqnarray*}
\Vert u\Vert_{\mathcal{X}^{s,p}_{\alpha, \beta}}&:=&\sup_{0<t<\infty}( t^{\alpha} \Vert u(t)\Vert_{H^s_p}+ t^{\beta}\Vert u_t(t)\Vert_{H_p^{s-1}}),
\end{eqnarray*}
 with $\alpha=\frac{1}{\lambda-1} [ 2-\theta-\frac n2(1-\frac 2p) ]$ and $\beta=\alpha+1-\theta.$

\begin{proposition}\label{nl2g} Let $\lambda\geq 3$ be a positive integer such that $\alpha_1(\lambda-2)>1,$  $\alpha_1=\frac{n+2(\theta-1)}{2},$  $\theta \in (\frac{2-n}{2},1]$ if $n=1,2$ and $\theta\in [0,1]$ if $n\geq 3,$ and consider
$s>\frac{n+4\theta -6}{2}$. Then, there exists $C_1>0$ such that
\begin{equation}
\left\Vert \int_{0}^{t} \Lambda_{\theta}(t-\tau)(|u|^{\lambda}-|\tilde{u}|^{\lambda})d\tau \right\Vert_{\mathcal{Y}^{s,\alpha_1}}\leq C_1 (1+t^{1-\theta}) \Vert u-\tilde{u}\Vert_{\mathcal{Y}^{s,\alpha_1}}(\Vert u\Vert^{\lambda-1}_{\mathcal{Y}^{s,\alpha_1\alpha_1}}+\Vert\tilde{u}\Vert^{\lambda-1}_{\mathcal{Y}^{s,\alpha_1}}),  \label{inl2za}
\end{equation}
where $\Lambda_{\theta}(t)=S(t)(I-\Delta)^{-1}(-\Delta)^{\theta}.$ 
\end{proposition}
\begin{proof} From Lemma \ref{HsL1(1+t)<=1} we get
\begin{align}
\left\Vert \int_{0}^{t} \Lambda_{\theta}(t-\tau)(|u|^{\lambda}-|\tilde{u}|^{\lambda})d\tau\right\Vert_{L^{\infty}}& \leq C \int_{0}^{t} (1+t-\tau)^{-\frac{n+2(\theta-1)}{2}}\left(\Vert(|u|^{\lambda}-|\tilde{u}|^{\lambda})\Vert_{L^{1}}+\Vert(|u|^{\lambda}-|\tilde{u}|^{\lambda})\Vert_{H^{s}}\right)d\tau \notag\\
&:=J_1+J_2.\label{b10ga}
\end{align}
Now, from Lemma \ref{leibnitz}, using that $\alpha_1(\lambda-2)>1$ and taking into account Lemma \ref{lem_int}, we bound the right hand side of (\ref{b10ga}) as follows
\begin{eqnarray}
J_1 &\leq &C\int_{0}^{t} (1+t-\tau)^{-\alpha_1} \left( \Vert u\Vert_{L^\infty}+ \Vert \tilde{u}\Vert_{L^\infty}\right)^{\lambda-2}( \Vert u\Vert_{L^2}+ \Vert \tilde{u}\Vert_{L^2})\Vert u-\tilde{u}\Vert_{L^2}d\tau\nonumber\\
&\leq & C\int_{0}^{t} (1+t-\tau)^{-\alpha_1}(1+\tau)^{-\alpha_1(\lambda-2)}\left( \Vert u\Vert_{\mathcal{Y}^{s,\alpha_1}}+ \Vert \tilde{u}\Vert_{\mathcal{Y}^{s,\alpha_1}}\right)^{\lambda-2}( \Vert u\Vert_{\mathcal{Y}^{s,\alpha_1}}+ \Vert \tilde{u}\Vert_{\mathcal{Y}^{s,\alpha_1}})\Vert u-\tilde{u}\Vert_{\mathcal{Y}^{s,\alpha_1}}d\tau\nonumber\\
&\leq &  C\left( \Vert u\Vert_{\mathcal{Y}^{s,\alpha_1}}+ \Vert \tilde{u}\Vert_{\mathcal{Y}^{s,\alpha_1}}\right)^{\lambda-1}\Vert u-\tilde{u}\Vert_{\mathcal{Y}^{s,\alpha_1}}(1+t)^{-\alpha_1},\label{r1a}
\end{eqnarray}
and
\begin{eqnarray}
J_2 &\leq &C\int_{0}^{t} (1+t-\tau)^{-\alpha_1}\left( \Vert u\Vert_{L^\infty}+ \Vert \tilde{u}\Vert_{L^\infty}\right)^{\lambda-2}( \Vert u\Vert_{H^s}+ \Vert \tilde{u}\Vert_{H^s})\Vert u-\tilde{u}\Vert_{L^\infty}d\tau\nonumber\\
&&+C\int_{0}^{t} (1+t-\tau)^{-\alpha_1}\left( \Vert u\Vert_{L^\infty}+ \Vert \tilde{u}\Vert_{L^\infty}\right)^{\lambda-1}\Vert u-\tilde{u}\Vert_{H^s}d\tau\nonumber\\
&\leq & C\int_{0}^{t} (1+t-\tau)^{-\alpha_1}(1+\tau)^{-\alpha_1(\lambda-1)}\left( \Vert u\Vert_{\mathcal{Y}^{s,\alpha_1}}+ \Vert \tilde{u}\Vert_{\mathcal{Y}^{s,\alpha_1}}\right)^{\lambda-1}\Vert u-\tilde{u}\Vert_{\mathcal{Y}^{s,\alpha_1}}d\tau\nonumber\\
&\leq &  C\left( \Vert u\Vert_{\mathcal{Y}^{s,\alpha_1}}+ \Vert \tilde{u}\Vert_{\mathcal{Y}^{s,\alpha_1}}\right)^{\lambda-1}\Vert u-\tilde{u}\Vert_{\mathcal{Y}^{s,\alpha_1}}(1+t)^{-\alpha_1}.\label{r2a}
\end{eqnarray}
Thus, from (\ref{b10ga}), (\ref{r1a}) and (\ref{r2a}), we conclude that
\begin{eqnarray}\label{DesPro1a}
&&\sup_{0<t<\infty}(1+t)^{\alpha_1}\left\Vert \int_{0}^{t} \Lambda_{\theta}(t-\tau)(|u|^{\lambda}-|\tilde{u}|^{\lambda})d\tau \right\Vert_{L^{\infty}} \leq C \Vert u-\tilde{u}\Vert_{\mathcal{Y}^{s,\alpha_1}}(\Vert u\Vert^{\lambda-1}_{\mathcal{Y}^{s,\alpha_1}}+\Vert\tilde{u}\Vert^{\lambda-1}_{\mathcal{Y}^{s,\alpha_1}}).
\end{eqnarray}
On the other hand, applying Corollary \ref{Coro1}, we arrive at 
\begin{eqnarray}\label{DesPro2a}
&&\vspace{-1cm}\left\Vert \int_{0}^{t} \Lambda_{\theta}(t-\tau)(|u|^{\lambda}-|\tilde{u}|^{\lambda})d\tau\right\Vert_{H^{s}} \leq C \int_{0}^{t} (t-\tau)^{1-\theta} \Vert |u|^{\lambda}-|\tilde{u}|^{\lambda}\Vert_{H^{s}}d\tau \leq C t^{1-\theta}  \int_{0}^{t} \Vert |u|^{\lambda}-|\tilde{u}|^{\lambda}\Vert_{H^{s}}d\tau\nonumber\\
&&\leq C t^{1-\theta}  \int_{0}^{t}( \left( \Vert u\Vert_{L^\infty}+ \Vert \tilde{u}\Vert_{L^\infty}\right)^{\lambda-2}( \Vert u\Vert_{H^s}+ \Vert \tilde{u}\Vert_{H^s})\Vert u-\tilde{u}\Vert_{L^\infty}+\left( \Vert u\Vert_{L^\infty}+ \Vert \tilde{u}\Vert_{L^\infty}\right)^{\lambda-1}\Vert u-\tilde{u}\Vert_{H^s})d\tau\nonumber\\
&&\leq C t^{1-\theta}  \left( \Vert u\Vert_{\mathcal{Y}^{s,\alpha_1}}+ \Vert \tilde{u}\Vert_{\mathcal{Y}^{s,\alpha_1}}\right)^{\lambda-1}\Vert u-\tilde{u}\Vert_{\mathcal{Y}^{s,\alpha_1}}\int_{0}^{t} (1+\tau)^{-\alpha_1(\lambda-1)}d\tau\nonumber\\
&&\leq C t^{1-\theta}  \left( \Vert u\Vert_{\mathcal{Y}^{s,\alpha_1}}+ \Vert \tilde{u}\Vert_{\mathcal{Y}^{s,\alpha_1}}\right)^{\lambda-1}\Vert u-\tilde{u}\Vert_{\mathcal{Y}^{s,\alpha_1}}.
\end{eqnarray}
Finally, from Corollary \ref{Coro2}, we obtain 
\begin{eqnarray}\label{DesPro3a}
&&\vspace{-1cm}\left\Vert \int_{0}^{t} \partial_t \Lambda_{\theta}(t-\tau)(|u|^{\lambda}-|\tilde{u}|^{\lambda})d\tau\right\Vert_{H^{s-1}}  \leq C \int_{0}^{t}\Vert |u|^{\lambda}-|\tilde{u}|^{\lambda}\Vert_{H^{s-1}}d\tau  \nonumber\\
&&\leq C   \int_{0}^{t}( \left( \Vert u\Vert_{L^\infty}+ \Vert \tilde{u}\Vert_{L^\infty}\right)^{\lambda-2}( \Vert u\Vert_{H^{s-1}}+ \Vert \tilde{u}\Vert_{H^{s-1}})\Vert u-\tilde{u}\Vert_{L^\infty}+\left( \Vert u\Vert_{L^\infty}+ \Vert \tilde{u}\Vert_{L^\infty}\right)^{\lambda-1}\Vert u-\tilde{u}\Vert_{H^{s-1}})d\tau\nonumber\\
&&\leq C  \left( \Vert u\Vert_{\mathcal{Y}^{s,\alpha_1}}+ \Vert \tilde{u}\Vert_{\mathcal{Y}^{s,\alpha_1}}\right)^{\lambda-1}\Vert u-\tilde{u}\Vert_{\mathcal{Y}^{s,\alpha_1}}\int_{0}^{t} (1+\tau)^{-\alpha_1(\lambda-1)}d\tau\nonumber\\
&&\leq C   \left( \Vert u\Vert_{\mathcal{Y}^{s,\alpha_1}}+ \Vert \tilde{u}\Vert_{\mathcal{Y}^{s,\alpha_1}}\right)^{\lambda-1}\Vert u-\tilde{u}\Vert_{\mathcal{Y}^{s,\alpha_1}}.
\end{eqnarray}
From (\ref{DesPro1a}),  (\ref{DesPro2a}) and  (\ref{DesPro3a}), we obtain the desire result. 
\end{proof}

\begin{proposition}\label{gbl2} Let $\lambda \geq 2,$ $\theta \in (\frac{2-n}{2},1]$ if $n=1,2$ and $\theta\in [0,1]$ if $n\geq 3,$ and $\frac{1}{\lambda}>\alpha>0.$ Assume $2\leq p \leq q\leq  \infty,$ $\frac n2(1-\frac 2p)<1$ and consider $\sigma,s$ such that $s>\sigma,$  $n(\frac{1}{p}-\frac{1}{q})\leq \sigma<3-n-2\theta$ and $(\frac{\lambda}{q}+\frac{1}{p}-1)\frac{n}{\lambda-1}+\sigma\leq s<\min\{\frac{n}{q},\lambda-1\}+\sigma.$ Then, there exists a constant $C_2>0$ such that
\begin{eqnarray}
\left\Vert \int_{0}^{t} \Lambda_{\theta}(t-\tau)(|u|^{\lambda}-|\tilde{u}|^{\lambda})d\tau \right\Vert_{{\mathcal{X}^{s}_{\alpha, \beta}}} \leq C_2\Vert u-\tilde{u}\Vert_{{\mathcal{X}^{s}_{\alpha, \beta}}} \left(\Vert u\Vert^{\lambda-1}_{{\mathcal{X}^{s}_{\alpha, \beta}}}+\Vert\tilde{u}\Vert^{\lambda-1}_{{\mathcal{X}^{s}_{\alpha, \beta}}}\right).  \label{inl2x}
\end{eqnarray}
\end{proposition}
\begin{proof} From Lemma \ref{LamThe<=1}, the embedding $H^s_{p}\subset H_{q}^{s-\sigma},$ Lemma \ref{NonIne1}, and recalling the integrability of the Beta function (notice that $1>\alpha\lambda,$ and $\alpha>0$ implies that $2>\theta+\frac{n}{2}(1-\frac{2}{p})$), we arrive at
\begin{eqnarray*}
&&\left\Vert \int_{0}^{t} \Lambda_{\theta}(t-\tau)(|u|^{\lambda}-|\tilde{u}|^{\lambda})d\tau \right\Vert_{H_p^{s}} \leq C \int_{0}^{t} (t-\tau)^{1-\theta-\frac n2(1-\frac 2p)}  \Vert (|u|^{\lambda}-|\tilde{u}|^{\lambda })\Vert_{H^{s-\sigma}_{p'}}d\tau\nonumber\\
&&\ \ \leq C \int_{0}^{t} (t-\tau)^{1-\theta-\frac n2(1-\frac 2p)} \Vert u-\tilde{u}\Vert_{H^{s-\sigma}_q}(\Vert u\Vert^{\lambda-1}_{H^{s-\sigma}_q}+\Vert \tilde{u}\Vert^{\lambda-1}_{H^{s-\sigma}_q})d\tau\nonumber\\
&&\ \ \leq C  \int_{0}^{t} (t-\tau)^{1-\theta-\frac n2(1-\frac 2p)}  \Vert u-\tilde{u}\Vert_{H^{s}_p}(\Vert u\Vert^{\lambda-1}_{H^{s}_p}+\Vert \tilde{u}\Vert^{\lambda-1}_{H^{s}_p})d\tau\nonumber\\
&&\ \ \leq C \Vert u-\tilde{u}\Vert_{{\mathcal{X}^{s}_{\alpha, \beta}}} [ \Vert u\Vert^{\lambda-1}_{{\mathcal{X}^{s}_{\alpha, \beta}}}+\Vert\tilde{u}\Vert^{\lambda-1}_{{\mathcal{X}^{s}_{\alpha, \beta}}} ]  \int_{0}^{t} (t-\tau)^{1-\theta-\frac n2(1-\frac 2p)}\tau^{-\alpha \lambda}  d\tau \nonumber\\
&&\ \ \leq C t^{2-\theta-\frac n2(1-\frac 2p)-\alpha\lambda} \Vert u-\tilde{u}\Vert_{{\mathcal{X}^{s}_{\alpha, \beta}}} [ \Vert u\Vert^{\lambda-1}_{{\mathcal{X}^{s}_{\alpha, \beta}}}+\Vert\tilde{u}\Vert^{\lambda-1}_{{\mathcal{X}^{s}_{\alpha, \beta}}} ]\nonumber\\
&&\ \ = C t^{-\alpha} \Vert u-\tilde{u}\Vert_{{\mathcal{X}^{s}_{\alpha, \beta}}} [ \Vert u\Vert^{\lambda-1}_{{\mathcal{X}^{s}_{\alpha, \beta}}}+\Vert\tilde{u}\Vert^{\lambda-1}_{{\mathcal{X}^{s}_{\alpha, \beta}}} ].
\end{eqnarray*}
Multiplying the last inequality by $t^{\alpha}$ we have
\begin{equation}\label{talpha}
t^{\alpha}\left\Vert \int_{0}^{t} \Lambda_{\theta}(t-\tau)(|u|^{\lambda}-|\tilde{u}|^{\lambda})d\tau \right\Vert_{H_p^{s}} \leq C \Vert u-\tilde{u}\Vert_{{\mathcal{X}^{s}_{\alpha, \beta}}} [ \Vert u\Vert^{\lambda-1}_{{\mathcal{X}^{s}_{\alpha, \beta}}}+\Vert\tilde{u}\Vert^{\lambda-1}_{{\mathcal{X}^{s}_{\alpha, \beta}}} ].
\end{equation}
On the other hand, from Lemma \ref{DtLamThe<=1}, the embedding $H^s_{p}\subset H_{q}^{s-\sigma},$ Lemma \ref{NonIne1}, and recalling the integrability of the Beta function, we get
\begin{eqnarray*}
&&\left\Vert \int_{0}^{t} \partial_t  \Lambda_{\theta} (t-\tau)(|u|^{\lambda}-|\tilde{u}|^{\lambda})d\tau \right\Vert_{H_p^{s-1}} \leq C \int_{0}^{t}  (t-\tau)^{-\frac n2(1-\frac 2p)}  \Vert (|u|^{\lambda}-|\tilde{u}|^{\lambda })\Vert_{H^{s-\sigma}_{p'}}d\tau\nonumber\\
&&\ \ \leq C \int_{0}^{t} (t-\tau)^{-\frac n2(1-\frac 2p)}   \Vert u-\tilde{u}\Vert_{H^{s-\sigma}_q}(\Vert u\Vert^{\lambda-1}_{H^{s-\sigma}_q}+\Vert \tilde{u}\Vert^{\lambda-1}_{H^{s-\sigma}_q})\nonumber\\
&&\ \ \leq C  \int_{0}^{t} (t-\tau)^{-\frac n2(1-\frac 2p)}   \Vert u-\tilde{u}\Vert_{H^{s}_p}(\Vert u\Vert^{\lambda-1}_{H^{s}_p}+\Vert \tilde{u}\Vert^{\lambda-1}_{H^{s}_p})\nonumber\\
&&\ \ \leq C  \Vert u-\tilde{u}\Vert_ {{\mathcal{X}^{s}_{\alpha, \beta}}}(\Vert u\Vert^{\lambda-1}_ {{\mathcal{X}^{s}_{\alpha, \beta}}} +\Vert \tilde{u}\Vert^{\lambda-1}_ {{\mathcal{X}^{s}_{\alpha, \beta}}})  \int_{0}^{t} (t-\tau)^{-\frac n2(1-\frac 2p)}  \tau^{-\alpha \lambda} d\tau \nonumber\\
&&\ \ \leq C  \Vert u-\tilde{u}\Vert_ {{\mathcal{X}^{s}_{\alpha, \beta}}}(\Vert u\Vert^{\lambda-1}_ {{\mathcal{X}^{s}_{\alpha, \beta}}} +\Vert \tilde{u}\Vert^{\lambda-1}_ {{\mathcal{X}^{s}_{\alpha, \beta}}})  t^{1-\frac n2(1-\frac 2p) -\alpha \lambda} \nonumber\\
&&\ \ = C  \Vert u-\tilde{u}\Vert_ {{\mathcal{X}^{s}_{\alpha, \beta}}}(\Vert u\Vert^{\lambda-1}_ {{\mathcal{X}^{s}_{\alpha, \beta}}} +\Vert \tilde{u}\Vert^{\lambda-1}_ {{\mathcal{X}^{s}_{\alpha, \beta}}})  t^{-\beta}.
\end{eqnarray*}
Multiplying by $t^{\beta},$ we obtain
\begin{equation}\label{tbeta}
t^{\beta}\left\Vert \int_{0}^{t} \partial_t  \Lambda_{\theta} (t-\tau)(|u|^{\lambda}-|\tilde{u}|^{\lambda})d\tau \right\Vert_{H_p^{s-1}} \leq C \Vert u-\tilde{u}\Vert_ {{\mathcal{X}^{s}_{\alpha, \beta}}}(\Vert u\Vert^{\lambda-1}_ {{\mathcal{X}^{s}_{\alpha, \beta}}} +\Vert \tilde{u}\Vert^{\lambda-1}_ {{\mathcal{X}^{s}_{\alpha, \beta}}}).
\end{equation}
From (\ref{talpha}) and (\ref{tbeta}), taking the supremum for $t>0,$ we obtain the desired result.
\end{proof}

\subsection{Proof of Theorem \ref{teo_global_2}}
\begin{proof}We consider the closed ball $$B_{\delta_1}=\{u\in C([0, \infty), H^{s}(\mathbb{R}^n))\cap C ((0,\infty);L^\infty(\mathbb{R}^n)) \cap C^1([0, \infty), H^{s-1}(\mathbb{R}^n)):\ \Vert u\Vert_{\mathcal{Y}^{s,\alpha_1}}\leq \delta_1\},\ \delta_1>0,$$  endowed with the complete metric $d(\cdot,\cdot),$ defined by $d(u,\tilde{u})=\Vert u-\tilde{u}\Vert_{\mathcal{Y}^{s,\alpha_1}}.$ Then, we prove that the map $ \Phi $ defined by
\[\Phi(u)=\partial_t S(t) u_0(x)+S(t)\Delta u_1(x) -\int_0^t  \Lambda(t-\tau)|u(x,\tau)|^{\lambda} d\tau,\]
is a contraction on $(B_{\delta_1},d).$ 
Let $\delta_1>0$ such that $2C_1\delta_1^{\lambda-1}<1,$ being $C_1$ the constant in (\ref{inl2za}). From the assumption on the initial data, Lemmas \ref{DatIni1} , Corollary \ref{CoroDat}, and Proposition \ref{nl2g} with $\tilde{u}=0,$ and $\theta=1,$ we obtain (for all $u\in B_\delta$) that
\begin{align*}
\Vert \Phi(u)\Vert_{\mathcal{Y}^{s,\alpha_1}}&\leq \Vert \partial_t S(t) u_0(x)+S(t)\Delta u_1(x)\Vert_{\mathcal{Y}^{s,\alpha_1}}+\left\Vert \int_0^t  \Lambda(t-\tau)|u(x,\tau)|^{\lambda} d\tau\right\Vert_{\mathcal{Y}^{s,\alpha_1}} \\
&\leq  \sup_{0<t<\infty}({(1+t)^{\frac n2}}(\Vert \partial_t S(t)u_0\Vert_{L^\infty}+\Vert \partial_t S(t)u_0\Vert_{H^s}+\Vert \partial^2_t S(t)u_0\Vert_{H^{s-1}})\\
&+ \sup_{0<t<\infty}({(1+t)^{\frac n2}}\Vert S(t)\Delta u_1\Vert_{L^{\infty}}+\Vert S(t)\Delta u_1\Vert_{H^s}+\Vert \partial_t S(t)\Delta u_1\Vert_{H^{s-1}}) +C_1 \Vert u\Vert^\lambda_{\mathcal{Y}^{s,\alpha_1}} \\
&\leq  \sup_{0<t<\infty}({}(\Vert u_0\Vert_{L^1}+\Vert u_0\Vert_{H^{s+1}}+\Vert \partial_t S(t)u_0\Vert_{H^s}+\Vert \partial^2_t S(t)u_0\Vert_{H^{s-1}})\\
&+ \sup_{0<t<\infty}(\Vert u_1\Vert_{L^{1}}+\Vert u_1\Vert_{H^{s+2}}+\Vert S(t)\Delta u_1\Vert_{H^s}+\Vert \partial_t S(t)\Delta u_1\Vert_{H^{s-1}})  +C_1 \Vert u\Vert^\lambda_{\mathcal{Y}^{s,\alpha_1}}\\
&\leq  C(\Vert u_0\Vert_{L^{1}}+\Vert u_0\Vert_{H^{s+1}}+\Vert u_1\Vert_{L^{1}}+\Vert u_1\Vert_{H^{s+2}})+C_1 \delta_1^\lambda\leq \delta_1.
\end{align*}
Thus, $ \Phi(B_{\delta_1})\subset B_{\delta_1}.$ Now, taking $u,\tilde{u}\in B_{\delta_1},$ from Proposition \ref{nl2g} we get
\begin{eqnarray}
\Vert  \Phi(u)- \Phi(\tilde{u})\Vert_{\mathcal{Y}^{s,\alpha_1}} &\leq& C_1 \Vert u-\tilde{u}\Vert_{\mathcal{Y}^{s,\alpha_1}}(\Vert u\Vert^{\lambda-1}_{\mathcal{Y}^{s,\alpha_1}}+\Vert\tilde{u}\Vert^{\lambda-1}_{\mathcal{Y}^{s,\alpha_1}})\nonumber\\
&\leq& 2C_1\delta_1^{\lambda-1}\Vert u-\tilde{u}\Vert_{\mathcal{Y}^{s,\alpha_1}}.
\end{eqnarray}
Since $2C_1\delta_1^{\lambda-1}<1,$ the map $\Phi$ is a contraction on $B_{\delta_1}$. Consequently, we have a unique
fixed point in $B_{\delta_1}$, which is the unique solution $u,$ satisfying $\Vert u\Vert_{\mathcal{Y}^{s,\alpha_1}}\leq \delta_1$ of the integral equation (\ref{e3}). Finally, the time continuity of the solution can be obtained in the standard way; therefore we omit it (see for instance Banquet and Villamizar-Roa \cite{Banquet2}). 
\end{proof}
\subsection{Proof of Theorem \ref{teo_global_2b}}
\begin{proof} The proof of Theorem \ref{teo_global_2b} is based on a fixed point argument. We prove that the map $ \Phi_\theta$ defined by
\[\Phi_\theta(u)=\partial_t S(t) u_0(x)+S(t)\Delta u_1(x) -\int_0^t S(t-\tau)( I-\Delta)^{-1} (-\Delta)^{\theta}|u(x,\tau)|^{\lambda} d\tau,\]
is a contraction on the closed ball $B_{\delta_2}$ defined by $$B_{\delta_2}=\{u\in C([0,\infty):H^s_p(\mathbb{R}^n))\cap C^1([0,\infty);H^{s-1}_p(\mathbb{R}^n)):\Vert u \Vert_{\mathcal{X}^{s,p}_{\alpha,\beta}}\leq {\delta_2}\},\ {\delta_2}>0,$$
endowed with the complete metric $d(\cdot,\cdot)$ defined by $d(u,\tilde{u})=\Vert u-\tilde{u}\Vert_{\mathcal{X}^{s,p}_{\alpha,\beta}}.$ Let $\delta_2>0$ such that $2C_2\delta_2^{\lambda-1}<1,$ where $C_2>0$ is the constant in (\ref{inl2x}). From the assumption on the initial data and Proposition \ref{gbl2} with $\tilde{u}=0,$ we get (for all $u\in B_{\delta_2}$)
\begin{eqnarray*}
\Vert \Phi_\theta(u)\Vert_{\mathcal{X}^{s,p}_{\alpha,\beta}}&\leq& \Vert \partial_t S(t) u_0+S(t)\Delta u_1\Vert_{\mathcal{X}^{s,p}_{\alpha,\beta}}+\left\Vert \int_0^t S(t-\tau)( I-\Delta)^{-1} (-\Delta)^{\theta}|u(\cdot,\tau)|^{\lambda} d\tau\right\Vert_{\mathcal{X}^{s,p}_{\alpha,\beta}}\nonumber\\
&\leq & \sup_{0<t<\infty}t^\alpha(\Vert \partial_t S(t) u_0\Vert_{H^s_p}+\Vert S(t)\Delta u_1\Vert_{H^s_p})\nonumber\\
&&+ \sup_{0<t<\infty}t^\beta(\Vert \partial^2_{t} S(t) u_0\Vert_{H^{s-1}_p}+\Vert \partial_tS(t)\Delta u_1\Vert_{H^{s-1}_p})+C_2 \Vert u\Vert^\lambda_{\mathcal{X}^{s,p}_{\alpha,\beta}}\nonumber\\
&\leq & {\frac{\delta_2}{2}+C_2\delta_2^\lambda\leq \delta_2.}
\end{eqnarray*}
Thus, $ \Phi_\theta(B_{\delta_2})\subset B_{\delta_2}.$ Now, taking $u,\tilde{u}\in B_{\delta_2},$ from Proposition \ref{gbl2} we get
\begin{eqnarray*}
\Vert  \Phi_\theta(u)- \Phi_\theta(\tilde{u})\Vert_{\mathcal{X}^{s,p}_{\alpha,\beta}} &\leq& C_2 \Vert u-\tilde{u}\Vert_{\mathcal{X}^{s,p}_{\alpha,\beta}}(\Vert u\Vert^{\lambda-1}_{\mathcal{X}^{s,p}_{\alpha,\beta}}+\Vert\tilde{u}\Vert^{\lambda-1}_{\mathcal{X}^{s,p}_{\alpha,\beta}})\nonumber\\
&\leq& 2C_2\delta_2^{\lambda-1}\Vert u-\tilde{u}\Vert_{\mathcal{X}^{s,p}_{\alpha,\beta}}.
\end{eqnarray*}
Since $2C_2\delta_2^{\lambda-1}<1,$ the map $\Phi$ is a contraction on $B_{\delta_2}$. Consequently, we have a unique
fixed point in $B_{\delta_2}$, which is the unique solution $u,$ satisfying $\Vert u\Vert_{\mathcal{X}^{s,p}_{\alpha,\beta}}\leq \delta_2$ of the integral equation (\ref{e3}). Finally, the time continuity of the solution can be obtained in the standard way; therefore we omit it (see for instance Banquet and Villamizar-Roa \cite{Banquet2}). 
\end{proof}

\section{Local solutions}
Before proving the existence of local solutions, we establish some estimates of the nonlinear term of integro-differential equation (\ref{e3}) in the norm 
\begin{eqnarray*}
\Vert u\Vert_{\mathcal{Z}^{s,p}_{T}}&:=&\sup_{0<t<T}t^{\frac{n}{2}(1-\frac{2}{p})}( \Vert u(t)\Vert_{H^s_p}+\Vert u_t(t)\Vert_{H_p^{s-1}}).
\end{eqnarray*}

\begin{proposition}\label{nl3b} 
 Let $\lambda \geq 2,$ $0<T<1,$ $\theta \in (\frac{2-n}{2},1]$ if $n=1,2$ and $\theta\in [0,1]$ if $n\geq 3.$ Assume $2\leq p \leq q\leq  \infty,$ and consider $\sigma,s$ such that $s>\sigma,$ $n(\frac{1}{p}-\frac{1}{q})\leq \sigma<3-n-2\theta,$ $(\frac{\lambda}{q}+\frac{1}{p}-1)\frac{n}{\lambda-1}+\sigma\leq s<\min\{\frac{n}{q},\lambda-1\}+\sigma$ and, $\frac{n}{2}(1-\frac{2}{p})\lambda<1.$ Then, there exists $C_3>0$ such that
\begin{eqnarray}
\left\Vert \int_{0}^{t} \Lambda_{\theta}(t-\tau)(|u|^{\lambda}-|\tilde{u}|^{\lambda})d\tau \right\Vert_{{\mathcal{Z}^{s,p}_{T}}} \leq C_3 T^{2-\theta-\frac n2(1-\frac 2p)\lambda}\Vert u-\tilde{u}\Vert_{{\mathcal{Z}^{s,p}_{T}}} \left(\Vert u\Vert^{\lambda-1}_{{\mathcal{Z}^{s,p}_{T}}}+\Vert\tilde{u}\Vert^{\lambda-1}_{{\mathcal{Z}^{s,p}_{T}}}\right).  \label{inl2xc}
\end{eqnarray}
\end{proposition}
\begin{proof} From Lemma \ref{LamThe<=1}, the embedding $H^s_{p}\subset H_{q}^{s-\sigma},$ Lemma \ref{NonIne1}, and recalling the integrability of the Beta function (notice that $1>\frac{n}{2}(1-\frac{2}{p})\lambda$ and $\lambda\geq 2,$ imply that $2>\theta+\frac{n}{2}(1-\frac{2}{p})$), we arrive at
\begin{eqnarray}
&&\left\Vert \int_{0}^{t} \Lambda_{\theta}(t-\tau)(|u|^{\lambda}-|\tilde{u}|^{\lambda})d\tau \right\Vert_{H_p^{s}} \leq C \int_{0}^{t} (t-\tau)^{1-\theta-\frac n2(1-\frac 2p)}  \Vert (|u|^{\lambda}-|\tilde{u}|^{\lambda })\Vert_{H^{s-\sigma}_{p'}}d\tau\nonumber\\
&&\ \ \leq C \int_{0}^{t} (t-\tau)^{1-\theta-\frac n2(1-\frac 2p)} \Vert u-\tilde{u}\Vert_{H^{s-\sigma}_q}(\Vert u\Vert^{\lambda-1}_{H^{s-\sigma}_q}+\Vert \tilde{u}\Vert^{\lambda-1}_{H^{s-\sigma}_q})d\tau\nonumber\\
&&\ \ \leq C  \int_{0}^{t} (t-\tau)^{1-\theta-\frac n2(1-\frac 2p)}  \Vert u-\tilde{u}\Vert_{H^{s}_p}(\Vert u\Vert^{\lambda-1}_{H^{s}_p}+\Vert \tilde{u}\Vert^{\lambda-1}_{H^{s}_p})d\tau\nonumber\\
&&\ \ \leq C (\Vert u-\tilde{u}\Vert_{{\mathcal{Z}^{s,p}_{T}}}) [ \Vert u\Vert^{\lambda-1}_{{\mathcal{Z}^{s,p}_{T}}}+\Vert\tilde{u}\Vert^{\lambda-1}_{{\mathcal{Z}^{s,p}_{T}}} ]  \int_{0}^{t} (t-\tau)^{1-\theta-\frac n2(1-\frac 2p)} \tau^{-\frac n2(1-\frac 2p)\lambda}d\tau \nonumber\\
&&\ \ \leq C t^{2-\theta-\frac n2(1-\frac 2p)(1+\lambda)} \Vert u-\tilde{u}\Vert_{{\mathcal{Z}^{s,p}_{T}}} [ \Vert u\Vert^{\lambda-1}_{{\mathcal{Z}^{s,p}_{T}}}+\Vert\tilde{u}\Vert^{\lambda-1}_{{\mathcal{Z}^{s,p}_{T}}} ].\nonumber
\end{eqnarray}
Therefore,
\begin{eqnarray}
\sup_{0<t<T}t^{\frac n2(1-\frac 2p)}\left\Vert \int_{0}^{t} \Lambda_{\theta}(t-\tau)(|u|^{\lambda}-|\tilde{u}|^{\lambda})d\tau \right\Vert_{H_p^{s}} \leq  C T^{2-\theta-\frac n2(1-\frac 2p)\lambda} \Vert u-\tilde{u}\Vert_{{\mathcal{Z}^{s,p}_{T}}} [ \Vert u\Vert^{\lambda-1}_{{\mathcal{Z}^{s,p}_{T}}}+\Vert\tilde{u}\Vert^{\lambda-1}_{{\mathcal{Z}^{s,p}_{T}}} ].\label{loc24}
\end{eqnarray}
On the other hand, from Lemma \ref{DtLamThe<=1},  the embedding $H^s_{p}\subset H_{q}^{s-\sigma}$ and Lemma \ref{NonIne1}, we get
\begin{eqnarray}
&&\left\Vert \int_{0}^{t} \partial_t  \Lambda_{\theta} (t-\tau)(|u|^{\lambda}-|\tilde{u}|^{\lambda})d\tau \right\Vert_{H_p^{s-1}} \leq C \int_{0}^{t}  (t-\tau)^{-\frac n2(1-\frac 2p)}  \Vert (|u|^{\lambda}-|\tilde{u}|^{\lambda })\Vert_{H^{s-\sigma}_{p'}}d\tau\nonumber\\
&&\ \ \leq C \int_{0}^{t} (t-\tau)^{-\frac n2(1-\frac 2p)}   \Vert u-\tilde{u}\Vert_{H^{s-\sigma}_q}(\Vert u\Vert^{\lambda-1}_{H^{s-\sigma}_q}+\Vert \tilde{u}\Vert^{\lambda-1}_{H^{s-\sigma}_q})\nonumber\\
&&\ \ \leq C  \int_{0}^{t} (t-\tau)^{-\frac n2(1-\frac 2p)}   \Vert u-\tilde{u}\Vert_{H^{s}_p}(\Vert u\Vert^{\lambda-1}_{H^{s}_p}+\Vert \tilde{u}\Vert^{\lambda-1}_{H^{s}_p})\nonumber\\
&&\ \ \leq C  (\Vert u-\tilde{u}\Vert_{{\mathcal{Z}^{s,p}_{T}}}) [ \Vert u\Vert^{\lambda-1}_{{\mathcal{Z}^{s,p}_{T}}}+\Vert\tilde{u}\Vert^{\lambda-1}_{{\mathcal{Z}^{s,p}_{T}}} ] \int_{0}^{t} (t-\tau)^{-\frac n2(1-\frac 2p)}  \tau^{-\frac n2(1-\frac 2p)\lambda}d\tau \nonumber\\
&&\ \ \leq Ct^{1-\frac n2(1-\frac 2p)(1+\lambda)}  (\Vert u-\tilde{u}\Vert_{{\mathcal{Z}^{s,p}_{T}}})  (\Vert u-\tilde{u}\Vert_{{\mathcal{Z}^{s,p}_{T}}}) [ \Vert u\Vert^{\lambda-1}_{{\mathcal{Z}^{s,p}_{T}}}+\Vert\tilde{u}\Vert^{\lambda-1}_{{\mathcal{Z}^{s,p}_{T}}} ]. \nonumber
\end{eqnarray}
Therefore, 
\begin{eqnarray}
\sup_{0<t<T}t^{\frac n2(1-\frac 2p)}\left\Vert \int_{0}^{t} \Lambda_{\theta}(t-\tau)(|u|^{\lambda}-|\tilde{u}|^{\lambda})d\tau \right\Vert_{H_p^{s-1}} \leq  C T^{1-\frac n2(1-\frac 2p)\lambda} \Vert u-\tilde{u}\Vert_{{\mathcal{Z}^{s,p}_{T}}} [ \Vert u\Vert^{\lambda-1}_{{\mathcal{Z}^{s,p}_{T}}}+\Vert\tilde{u}\Vert^{\lambda-1}_{{\mathcal{Z}^{s,p}_{T}}} ].\label{loc21}
\end{eqnarray}
From (\ref{loc24}) and (\ref{loc21}) we obtain the desired result.
\end{proof}

\subsection{Proof of Theorem \ref{teo_local_1}}
\begin{proof} The proof of Theorem \ref{teo_local_1} is based on a fixed point argument. We will prove that the map $ \Phi_\theta$ defined by
\[\Phi_\theta(u)=\partial_t S(t) u_0(x)+S(t)u_1(x) -\int_0^t S(t-\tau)( I-\Delta)^{-1} (-\Delta)^{\theta}|u(x,\tau)|^{\lambda} d\tau,\]
is a contraction on the closed ball $B_R,$ defined by $$B_R=\{u\in C([0,T], H_p^{s}(\mathbb{R}^n)) \cap C^1([0,T], H_p^{s-1}(\mathbb{R}^n)):\Vert u \Vert_{\mathcal{Z}^{s,p}_T}\leq R\},\ R>0.$$
 From Lemma \ref{DatIni2} and Proposition \ref{nl3b} with $\tilde{u}=0,$ we get
\begin{eqnarray}
\Vert \Phi_\theta(u)\Vert_{\mathcal{Z}^{s,p}_T}&\leq& \Vert \partial_t S(t) u_0(x)+S(t)u_1(x)\Vert_{\mathcal{Z}^{s,p}_T}+\left\Vert \int_0^t S(t-\tau)( I-\Delta)^{-1} (-\Delta)^{\theta}|u(x,\tau)|^{\lambda} d\tau\right\Vert_{\mathcal{Z}^{s,p}_T}\nonumber\\
&\leq & \sup_{0<t<T}t^{\frac{n}{2}(1-\frac{2}{p})}(\Vert \partial_t S(t)u_0\Vert_{H_p^s}+\Vert \partial^2_t S(t)u_0\Vert_{H_p^{s-1}})\nonumber\\
&&+ \sup_{0<t<T}t^{\frac{n}{2}(1-\frac{2}{p})}(\Vert S(t)\Delta u_1\Vert_{H_p^s}+\Vert \partial_t S(t)\Delta u_1\Vert_{H_p^{s-1}})\nonumber\\
&&+C_3T^{2-\theta-\frac{n}{2}(1-\frac{2}{p})\lambda}R^\lambda\nonumber\\
&\leq & C(\Vert u_0\Vert_{H_{p'}^{s+1-\sigma}}+\Vert u_1\Vert_{H_{p'}^{s+2-\sigma}})+C_3T^{2-\theta-\frac{n}{2}(1-\frac{2}{p})\lambda}R^\lambda.\label{zz22}
\end{eqnarray}
Now we take $R=2C(\Vert u_0\Vert_{H_{p'}^{s+1-\sigma}}+\Vert u_1\Vert_{H_{p'}^{s+2-\sigma}})>0$ y $T>0$ such that $$C_3T^{2-\theta-\frac{n}{2}(1-\frac{2}{p})\lambda}R^{\lambda-1}<\frac{1}{2}.$$
Thus, from (\ref{zz22}) we get $\Vert \Phi_\theta(u)\Vert_{\mathcal{Z}^{s,p}_T}\leq \frac{R}{2}+\frac{R}{2}=R,$ for $u\in B_R,$ that is, $\Phi_\theta(B_R)\subseteq B_R.$ On the other hand, using again Proposition \ref{nl3b}, we also have
\[
\Vert \Phi_\theta(u)- \Phi_\theta(\tilde{u})\Vert_{\mathcal{Z}^{s,p}_T}\leq C_3 T^{2-\theta-\frac n2(1-\frac 2p)\lambda}2R^{\lambda-1}\Vert u-\tilde{u}\Vert_{{\mathcal{Z}^{s,p}_{T}}},
\]
for all $u,\tilde{u}\in B_R$ . Consequently, $ \Phi_\theta$ is a contraction in $B_R$ and then the Banach fixed point
theorem ensures the existence of a unique solution $u\in \mathcal{Z}^{s,p}_T$ of  (\ref{NorPlaEqu}). Finally, the time continuity of the solution can be obtained in the standard way; therefore we omit it (see for instance Banquet and Villamizar-Roa \cite{Banquet2}). 
\end{proof}

\end{document}